\documentclass[opre,nonblindrev]{informs3} 

\DoubleSpacedXI 

\usepackage{geometry}
\usepackage{algpseudocode}
\usepackage{algorithm}
\algnewcommand{\IIf}[1]{\State\algorithmicif\ #1\ \algorithmicthen}
\algnewcommand{\EndIIf}{\unskip\ \algorithmicend\ \algorithmicif}
\algtext*{EndIf}

\usepackage{endnotes}
\usepackage{geometry}
\usepackage{graphicx}
\usepackage{color}
\usepackage{float}
\usepackage{caption}%
\usepackage[caption=false,farskip=0pt,labelfont={bf}]{subfig} 
\usepackage{diagbox}

\let\footnote=\endnote
\let\enotesize=\normalsize
\def\notesname{Endnotes}%
\def\makeenmark{$^{\theenmark}$}
\def\enoteformat{\rightskip0pt\leftskip0pt\parindent=1.75em
  \leavevmode\llap{\theenmark.\enskip}}

\usepackage{enumerate}
\usepackage{natbib}
\usepackage{epsfig,epsf,fancybox}

 \bibpunct[, ]{(}{)}{,}{a}{}{,}%
 \def\bibfont{\small}%

  \newcommand{\Ex}{\mathbb{E}} 
\newcommand{\R}{\mathbb{R}}

\newcommand{\F}{\mathbb{F}}
\newcommand{\Prob}{\mathbb{P}}
\newcommand{\bz}{\mathbf{z}}
\newcommand{\bx}{\mathbf{x}}

\newcommand{\bp}{\mathbf{p}}

\TheoremsNumberedThrough     
\ECRepeatTheorems
\newenvironment{proof}{\paragraph{Proof:}}{\hfill$\square$}

\EquationsNumberedThrough    

\renewcommand{\theARTICLETOP}{}
\begin{document}




\TITLE{A Unified Framework for Generalized Moment Problems:
a Novel Primal-Dual Approach}

\ARTICLEAUTHORS{
\AUTHOR{Jiayi Guo, Simai He, Bo Jiang, Zhen Wang}
\AFF{School of Information Management and Engineering, Shanghai University of Finance and Economics, China, 200433,
guo.jiayi@sufe.edu.cn,
simaihe@mail.shufe.edu.cn,
isyebojiang@gmail.com,
zhenwang@163.sufe.edu.cn
}
}


\ABSTRACT{
Generalized moment problems optimize functional expectation over a class of distributions with generalized moment constraints, i.e., the function in the moment can be any measurable function. 
These problems have recently attracted growing interest due to their great flexibility in representing nonstandard moment constraints, such as geometry-mean constraints, entropy constraints, and exponential-type moment constraints. Despite the increasing research interest, analytical solutions are mostly missing for these problems, and researchers have to settle for nontight bounds or numerical approaches that are either suboptimal or only applicable to some special cases. 
In addition, the techniques used to develop closed-form solutions to the standard moment problems are tailored for specific problem structures.
In this paper, we propose a framework that provides a unified treatment for any moment problem. The key ingredient of the framework is a novel primal-dual optimality condition. 
This optimality condition enables us to 
reduce the original infinite dimensional problem to a nonlinear equation system with a finite number of variables.
In solving three specific moment problems, the framework demonstrates a clear path for identifying the analytical solution if one is available, otherwise, it produces semi-analytical solutions that lead to efficient numerical algorithms. Finally, through numerical experiments, we provide further evidence regarding the performance of the resulting algorithms by solving a moment problem and a distributionally robust newsvendor problem.
}%
\maketitle
%
\section{Introduction}
 The generalized moment problem {(see, e.g., \citealt{bertsimas2005optimal})}
aims to optimize the expectation of a measurable function $g(\cdot)$ with general distributional moment information $\Ex[h_i(X)]=m_i$ for $i=1,\dots,n$, where $h_i(\cdot)$ can be {\it any measurable} function. 
The study of this problem dates back to the pioneering work of \cite{cheb} and \cite{markov1884certain} on the {\it standard} moment problem (i.e., $h_i(X) = X^i$).  Despite the long history, the research on this fundamental problem is still active today
in modern probability theory \citep{smith1995generalized, bertsimas2005optimal,he2010bounding}. Moreover, the generalized moment problem is viewed as 
an important building block for distributionally robust optimization, with
wide applications in inventory control \citep{1958A, perakis2008regret,natarajan2018asymmetry, das2021heavy}, portfolio optimization \citep{ghaoui2003worst, bertsimas2010models, delage2010distributionally,zymler2013worst,rujeerapaiboon2016robust}, and statistical  learning \citep{lanckriet2002robust, mehrotra2014models,fathony2018distributionally}.

For the standard moment problem, the classic numerical algorithm comes from \cite{bertsimas2005optimal}, who formulate the dual of the problem as a polynomial-time solvable semidefinite program (SDP). For certain more structured problems, the relative entropy formulation {(see, e.g., \citealt{das2021heavy})} can be applied to further accelerate this solving process. However, methods that can yield \text{\it closed-form} solutions and provide management insights are often preferred. 
A classic example of such methods is due to \cite{1958A}, who obtains an analytical solution for the robust newsvendor problem, given the first two moments. Following this approach, a large volume of literature has discovered more closed-form decisions for some variants of Scarf's model by considering the asymmetry of demands \citep{natarajan2018asymmetry}, heavy-tailed distributed demands \citep{das2021heavy}, risk-averse objectives \citep{han2014risk}, etc.
The study of closed-form solutions is also prevalent in probability theory \citep{bertsimas2005optimal, he2010bounding, roos2021tight} and portfolio selection  \citep{ghaoui2003worst, zuluaga2009third, natarajan2010tractable, chen2011tight, li2018closed}.
However, the techniques that have been used in the literature to develop these closed forms are mathematically different and are only tailored for specific problems. In particular, much effort has been devoted to verifying the optimality of  analytical solution candidates that are provided without much explanation. 
However, discussion of how to come up with such analytical forms in the first place is largely missing. Thus, there is an urgent need for a unified framework that can systematically derive the closed-form solution for a given moment problem. 

The generalized moment problem  has recently attracted growing interest
due to its great flexibility in representing nonstandard moments,\footnotemark   \footnotetext{In this paper, we call a moment nonstandard if it cannot be represented as the expectation of some piece-wise polynomial function.}
such as geometry-mean constraints (i.e., $h_i(X)=\log {X}$)  in the pricing problem \citep{tamuz2013lower,elmachtoub2021value}, entropy constraints \citep{chen2019distributionally}, and exponential-type moment constraints (i.e., $h_i(X)=e^{X}$) including moment-generating constraints and sub-Gaussian distribution constraints \citep{honorio2014tight}.
Such exponential-type moment information is known to be crucial for establishing sharp tail estimations with large deviations  when a distribution has light tails. In contrast, geometry-mean constraints  are useful for estimating tail bounds with small deviations, especially when the  distribution has a heavy tail.
However, to the best of our knowledge, 
 the literature lacks an efficient algorithm that can exactly solve the generalized moment problem. The existing solution methods are either suboptimal \citep{smith1995generalized} or restricted to special cases \citep{popescu2007robust,chen2019distributionally}.
\subsection{Our Contributions}
To address the two issues raised above, we
propose a unified framework for solving the generalized moment problem, which is our main contribution. The merits of this framework can be summarized as follows:
\begin{enumerate}
    \item It provides a unified treatment for any generalized moment problem. Such treatment proves
useful in identifying an analytical solution of the problem if one is available. Even when a closed-form solution is unavailable, this framework can quickly reduce the problem to a system of equations with a small number of variables that  include only the optimal supporting points.
 For certain structured problems, this system of equations can lead to
numerical methods that are even more efficient than the SDP approach shown in our numerical experiments.
    
    \item The framework has great flexibility for accommodating nonstandard moment constraints/objective functions. It complements the current literature by providing an approach for finding the {\it exact} solutions of generalized moment problems. As an illustration, in Section \ref{sec:P-1et} we demonstrate how to solve an exponential moment constrained problem with this framework. 
    
    \item The framework provides 
    a novel viewpoint for generalized moment problems and lends alternative insights even for some well-studied models. For instance, we manage to obtain an exact semi-closed-form solution for the $1$st and $t$-th moment problem, which tightens the semi-closed-form bounds given in \cite{das2021heavy}, Proposition 3.4, under the same condition. It also enables us to find infinitely many optimal solutions in a degenerate case for the moment problem with the upper partial moment  objective and constraint in \cite{han2014risk}. 
 
\end{enumerate}
Apart from the framework, we also propose a novel optimality condition for the generalized moment problem by incorporating the primal feasible condition, the complementary slackness condition, and a tangent condition. To the best of our knowledge, the tangent condition has not been systematically treated as an optimality condition prior to our work, even though it has been implicitly used in the derivation of theoretical results (see, e.g., \citealt{smith1995generalized,popescu2007robust,zuluaga2009third}).
In fact, our 
optimality condition is equivalent to a sizeable
nonlinear system with an equal number of variables and equations, which paves the way for the analytical solution and efficient numerical algorithms for the problem.
\subsection{Related Literature}
Due to the popularity of closed-form solutions, in this subsection, we first review the vast literature
on the analytical solutions of the standard moment problems in the fields of inventory control, finance, and probability theory.
The moment problems in inventory control theory are mostly discussed in the literature on the distributionally robust newsvendor model. The previously mentioned 
seminal work of \cite{1958A} has been criticized for its over-conservative behavior, and several variants have been proposed and studied to address this issue. For instance, \cite{yue2006expected} derive analytical bounds on the min-max regret objective with mean-variance moment constraints. \cite{han2014risk} extend Scarf’s closed-form formula to the risk- and ambiguity-averse newsvendor problem. \cite{natarajan2018asymmetry} provide a closed-form expression for the worst-case newsvendor profit with mean, variance, and semivariance information, and they show that the worst case occurs in a three-point distribution. Recently, \cite{das2021heavy} have considered the problem with the first and $t$-th moment constraints that capture the heavy-tailed behavior of the demand, and they manage to derive a closed-form worst-case distribution when the order quantities are below a certain value.

Regarding the research on analytically solving the moment problems in finance, \cite{ghaoui2003worst} consider the distributionally robust single-period portfolio selection and obtain closed-form solutions of the worst-case value at risk over a mean-variance constraint. 
\cite{chen2011tight} extend this result to the disutility function in the form of a conditional value at risk (CVaR), or the form of lower partial moments. Under the same moment constraint, \cite{natarajan2010tractable} and \cite{li2018closed} identify  closed-form solutions for the optimized certainty equivalent risk measure in \cite{2010bental} and the law-invariant risk measure that is the most important extension of CVaR, respectively. 
 \cite{zuluaga2009third} analytically find the worst-case payoff of the European call option using up to the third-order moment, which characterizes the skewness of the asset return.

 In probability theory, moment problems have been applied to derive tight closed-form tail probability bounds, given the first three moments \citep{bertsimas2005optimal}. Subsequently, \cite{he2010bounding} extend this result to the problem with first-, second-, and fourth-order moments. Recently, \cite{roos2021tight} provide alternative tight lower and upper bounds on the tail probability under a bounded support, given the mean and mean absolute deviation of the random variable.
 
In contrast to the fruitful research on the standard moment problems, studies of nonstandard moment problems are quite limited, and they all focus on the numerical approaches.
\cite{smith1995generalized} applies an approximate procedure to reduce the dimension of the sample space of the problem by grid search. Unfortunately, the efficiency of this approach has not been supported by theoretical foundations or by extensive numerical experiments. For the moment problem with entropy constraints that can be characterized by tractable conic inequalities, \cite{chen2019distributionally}
propose a greedy improvement procedure that sequentially optimizes tractable relaxed subproblems. If we relax the objective to be the so-called one- or two-point support functions
while restricting the constraint to be the first two moments, \cite{popescu2007robust} successfully reduces this problem to a deterministic parametric quadratic program. 

In this paper, we focus on providing a unified treatment for various moment problems, including nonstandard ones, and on identifying the analytical solution when it is available.
We note that, even for the standard moment problems, to the best of our knowledge, a similar framework has not yet been proposed, as the methods described above for obtaining the analytical solutions are problem-dependent. 
There are also other framework-like methods that can solve a wide class
of moment problems have been proposed \citep{smith1995generalized, bertsimas2005optimal}. However, the classical SDP method \citep{bertsimas2005optimal} is a numerical approach, and it only works for standard moment problems, while \cite{smith1995generalized}'s method can be applied to generalized moment problems, but, as noted above, its efficiency has not been  theoretically or numerically justified. For these reasons, our work is an excellent complement to the existing literature.
\subsection{Outline and Notations}

The remainder of the paper is organized as follows. In Section \ref{sec:framework}, we present a novel primal-dual optimality condition for the generalized moment problems and propose a three-step unified framework for these problems.
In the subsequent sections, we analyze three concrete moment problems with our  framework. In Section \ref{Sec:1t}, we consider the $1$st and $t$-th moment problem that was proposed in \cite{das2021heavy}, and we obtain the same analytical solution under the same condition. In cases where an analytic solution is unavailable, our framework also provides a semi-analytical form of the optimal solution.
The merit of this form is that it relies on only one parameter that 
is a root of a nonlinear equation. Section \ref{sec:P-121} is devoted to the problem of minimizing $2$nd-order upper partial moments with the $1$st-order upper partial moment constraint. Not only do we obtain the same explicit optimal value as \cite{han2014risk}, but we also identify more analytical solutions for a degenerate case (see Lemma \ref{lm: p-z-121-2}). To demonstrate our framework's ability to handle nonpolynomial moments, we consider an exponential moment constrained problem in Section \ref{sec:P-1et}. Similar to Section \ref{Sec:1t}, we obtain an analytical solution and a semi-analytical solution for two scenarios that are defined by the range of the moment parameters. By solving those three moment problems, we show that our framework is capable of finding the closed-form solution if one exists. In addition, when a closed-form solution is not available, our framework leads to some efficient numerical algorithms. In Section \ref{sec:numerics}, we apply those algorithms to the $1$st and $t$-th moment problem and the distributionally robust newsvendor problem with an exponential moment ambiguity set as two illustrative examples. We show that the resulting algorithms indeed solve the two problems efficiently and thus demonstrate the benefit of our framework.

Throughout this paper, we denote vectors and matrices by boldface lowercase letters and capital letters, respectively. For probability distributions, we use $\mathbb{M}(\Omega)$ to denote the set of all Borel probability distributions on the support $\Omega \subseteq \R$, and we use $\{\bx; \bp\}_{D}$ to denote a discrete distribution with support $\bx = (x_1,x_2,..,x_K)^{T}\in \Omega^k$ and probability vector $\bp = (p_1,p_2,..,p_K)^{T}$ such that $p_i=\Prob(X=x_i)$ for $i=1,2,\ldots,K$.
Finally, we use $[x]_+$ to represent the positive function $\max(x,0)$.
\section{The Unified Framework for the Generalized Moment Problem}\label{sec:framework}
\subsection{The Generalized Moment Problem and the Optimality Condition}
We consider the generalized moment problem in the following form:
\begin{equation}\label{Prob:GMP}
\tag{\sf GMP}
	\begin{aligned}
	&&\quad \quad  Z_P=\max_{F(\cdot)}  & \int_{\Omega}g(x)\cdot \mathrm{d}F(x)\\
	&& \text{s.t. } & \int_{\Omega} h_i(x) \cdot \mathrm{d}F(x) = m_i,\;\; i=0,1,\ldots,n,\\
	&& & \mathrm{d}F(x) \ge 0,\;\; \forall x\in \Omega,
	\end{aligned}
\end{equation}	
with $h_0(x)=m_0=1$ and $\Omega \subseteq \R$. In contrast to the standard moment problem, where the moment constraints are defined by the expectations of certain piece-wise polynomial functions, the function $h_i(\cdot)$ in \eqref{Prob:GMP} can be {\it any measurable} function with respect to $\Omega$ for $i=1,2,\ldots,n$. Note that problem \eqref{Prob:GMP} can be viewed as a semi-infinite linear program with $n+1$ constraints. According to \cite{smith1995generalized} and \cite{bertsimas2005optimal}, there exists an optimal distribution of \eqref{Prob:GMP} with at most $n+1$ mass points.
Therefore, we consider an optimal distribution $\{\bx; \bp\}_{D}$
with finite support $\bx = (x_1,x_2,..,x_K)^{T}$ and probability vector $\bp = (p_1,p_2,..,p_K)^{T}$. Then
the moment constraints in \eqref{Prob:GMP} can be rewritten as
\begin{equation} \label{eq:primal}
\tag{\sf PriCond}
 \begin{pmatrix}
h_0(x_1)& \dots & h_0(x_K)\\
\vdots  & \ddots & \vdots\\
h_n(x_1)& \dots & h_n(x_K)
\end{pmatrix}\begin{pmatrix}
p_1\\
\vdots  \\
p_K\\
\end{pmatrix}=\begin{pmatrix}
m_0\\
\vdots  \\
m_n\\
\end{pmatrix} .
\end{equation}	
To study \eqref{Prob:GMP} from an alternative perspective, we consider the dual problem:
\begin{equation}
\label{Prob:DGMP}
\tag{\sf DGMP}
	\begin{aligned}
	&&\quad \quad  Z_D  =\min_{\bz \in \R^{n+1}} &  \sum\limits_{i=0}^{n}   z_i \cdot m_i\\
	&& \text{s.t. }   H(x;\mathbf{z}) &:=\sum_{i=0}^{n}z_i \cdot h_i(x)-g(x) \geq 0, \forall x \in \Omega.\\
	\end{aligned}
\end{equation}
Let $\mathbf{m}=(m_0,m_1,...,m_n)^T$  be a moment vector. The strong duality ($Z_P=Z_D$) holds if $\mathbf{m}$ lies in the interior of the set of all moment vectors that make problem \eqref{Prob:GMP} feasible \citep{bertsimas2005optimal}. It is well known that a primal feasible distribution $\{\bx; \bp \}_{D}$ and a dual feasible solution $\bz$ are optimal if and only if the complementary slackness condition holds \citep{smith1995generalized}, i.e., $p_j \cdot H(x_j ; \bz) =0$ for $\; \forall \;j=1,\cdots,K$. Since $p_j > 0$, we further have $H(x_j ; \bz) =0$ for  $\forall \;j=1,\cdots,K$, {which is equivalent to}:
\begin{equation} \label{eq:cs}
\tag{\sf SlackCond}
\begin{pmatrix}
h_0(x_1)& \dots & h_n(x_1)\\
\vdots & \ddots & \vdots\\
h_0(x_K)& \dots & h_n(x_K)
\end{pmatrix}\begin{pmatrix}
z_0\\
\vdots  \\
z_{n}\\
\end{pmatrix}=\begin{pmatrix}
g(x_1)\\
\vdots  \\
g(x_{K})
\end{pmatrix}.
\end{equation}
The equations above together with the dual feasibility imply that
$x_1, \ldots, x_K$ are global minimizers of the problem $\min_{x \in \Omega} H(x;\mathbf{z})$. Furthermore, if $x_j$ is the {\it differentiable interior point}, i.e., the differentiable point of $H(x;\mathbf{z})$ lying in the interior of $\Omega$, for some $1\le j \le K$, then according to the first-order optimality condition, we have 
$$H^{\prime}(x_j;\bz)=\sum_{i=0}^n z_i h^{\prime}_i(x_j)-g^{\prime}(x_j)=0.$$
Without loss of generality, we assume that $x_1, \ldots, x_k$ with $k \le K$ are the differentiable interior points. The {above equality} leads to the so-called \text{\it tangent condition}:
\begin{equation} \label{eq:Deriv}
 \tag{\sf TagntCond}
\begin{aligned}
&    & 
\begin{pmatrix}
{h^{\prime}_0(x_1)}& \dots & {h^{\prime}_n(x_1)}\\
\vdots   & \ddots & \vdots\\
{h^{\prime}_0(x_{k})}& \dots &{h^{\prime}_n(x_{k})}\\
\end{pmatrix}
\begin{pmatrix}
z_0\\
\vdots  \\
z_{n}\\
\end{pmatrix}=\begin{pmatrix}
{g^{\prime}(x_1)}\\
\vdots\\
{g^{\prime}(x_{k})}
\end{pmatrix}
\end{aligned}.
\end{equation}
Consequently, we formally {propose the three systems of}
equations \eqref{eq:primal}, \eqref{eq:cs}, and \eqref{eq:Deriv} as an optimality condition for problem \eqref{Prob:GMP} in the following theorem.
\begin{theorem} \label{tm: gpm}
 Suppose that strong duality holds between \eqref{Prob:GMP} and \eqref{Prob:DGMP}. Then the distribution $\{\bx \in \Omega^K; \bp \}_{D}$ and vector $\bz \in \R^{n+1}$ are optimal solutions to \eqref{Prob:GMP} and \eqref{Prob:DGMP}, respectively, if and only if the following optimality condition holds:
 \begin{enumerate}[i)]
 \item $\bp \geq 0$, and $H(x;\bz) \geq 0$ for all $x \in \Omega$;
 \item {\it the primal condition} \eqref{eq:primal} holds;
 \item {\it the complementary slackness condition} \eqref{eq:cs} holds;
 \item {\it the tangent condition} \eqref{eq:Deriv} is valid for all differentiable interior points in $\bx$.
 \end{enumerate}
\end{theorem}
\begin{proof}
The necessity of the optimality condition follows directly from the discussion immediately preceding this theorem. To prove the sufficiency of the condition, suppose that there exist a distribution $\{\bx \in \Omega^K ; \bp \}_{D}$ with $\bp \geq 0$ and a vector $\bz \in \R^{n+1}$ satisfying $H(x;\bz) \geq 0$ for all $x \in \Omega$ such that \eqref{eq:primal}, \eqref{eq:cs}, and \eqref{eq:Deriv} hold. Then $\{\bx; \bp\}_{D}$ is a feasible solution to the primal problem \eqref{Prob:GMP}, due to $\bp \geq 0$ and condition \eqref{eq:primal}, and $\bz$ is also a feasible solution to the dual problem \eqref{Prob:DGMP}, as $H(x;\bz) \geq 0$ for all $x \in \Omega$. Moreover, the complementary slackness is guaranteed by \eqref{eq:cs}. Therefore, $\{\bx \in \Omega^K; \bp \}_{D}$ and $\bz$ are optimal solutions to \eqref{Prob:GMP} and \eqref{Prob:DGMP}, respectively.
\end{proof}

We remark that \eqref{eq:primal} and \eqref{eq:cs} are standard conditions that are used to characterize the optimal solution. However, to the best of our knowledge, \eqref{eq:Deriv} has not been formally proposed as an optimality condition, although it has been implicitly used in the literature for the analysis of some structured moment problems (see, e.g., \citealt{smith1995generalized,popescu2007robust,zuluaga2009third}). In fact, condition \eqref{eq:Deriv} has a very clear geometric explanation, which is that $\sum_{i=0}^{n}z_i \cdot h_i(x)$ and $g(x)$ share the same tangent plane at the points $x_1, \cdots, x_k$, while it is possible that
such a tangent property does not hold for other intersecting points $x_j$ with $k<j\le K$, as illustrated in Figure \ref{Fig:4thOrder}.
Specifically, the instance in Figure \ref{Fig:4thOrder} is taken from the dual problem of upper bounding the probability under the first, second, and fourth moments \citep{he2010bounding}. In this case, $\sum_{i=0}^{n}z_i \cdot h_i(x)= z_0+z_1 x+z_2x^2+z_4x^4$ is a quartic function, and $g(x) = {1} _{x \geq q}(x)$ for some given constant $q$ is an indicator function. From the figure, we can see that ${1} _{x \geq q}(x)$ is tangent to $z_0+z_1 x+z_2x^2+z_4x^4$ at $u$ and $v$, where condition \eqref{eq:Deriv} holds, while the two functions intersect but are not tangent at point $q$, as ${1} _{x \geq q}(x)$ is not differentiable at this point.
\begin{figure} 
  \centering
  \includegraphics[height=0.4\textheight]{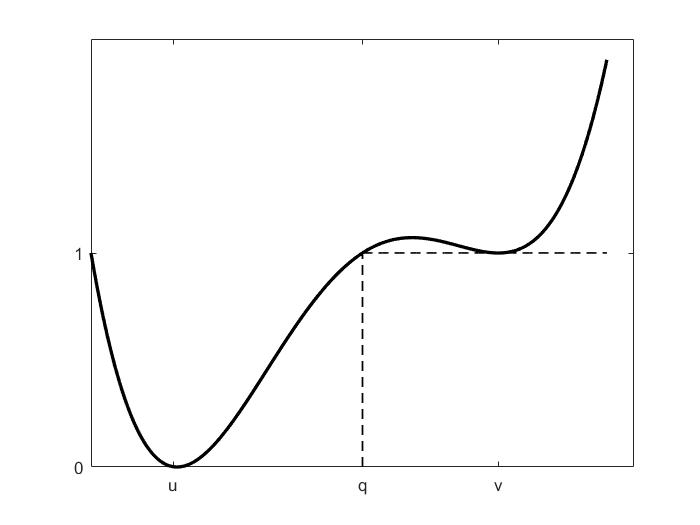}
  \caption{A quartic function shown above is an indicator function with three touching points $\{u,q,v\}$}
  \label{Fig:4thOrder}
\end{figure} 
\subsection{The Unified Framework}
In accordance with the optimality conditions \eqref{eq:primal}, \eqref{eq:cs}, and \eqref{eq:Deriv} stated in Theorem \ref{tm: gpm}, we propose the following three-step unified framework for solving the GMP.
 \vspace{0.2cm}
 	
 \begin{center}
 			\shadowbox{\begin{minipage}{5.5in}
			\textbf{The Unified Framework for Solving the GMP:}
			\begin{enumerate}
				\item {\it Identify the {\it rough} structure of the support of the optimal distribution}.
				\item {\it Provide the possible (semi-)analytical form of the optimal supporting points}.
				\item {\it Solve and verify the optimal distribution}.
			\end{enumerate}
			\end{minipage}}
\end{center}
The {\it rough} structure in Step 1 of the framework means the cardinality of the optimal support and the description of how the associated supporting points are distributed among {\it differentiable pieces.}\footnotemark  
\footnotetext{For the instance in Figure 1, there are two differential pieces, $(-\infty, q)$ and $(q,\infty)$. Therefore, the two tangent points $\{u, v\}$ can possibly be allocated in those two pieces in three ways: 
$\{u,q,v\},\{u,v,q\}$, and $\{q,u,v\}$. }
$\hspace{2mm}$ Each piece is defined as the interval connecting two consecutive non-differentiable points of the function $H(x;\mathbf{z})$, with $\pm \infty$ being treated as non-differentiable points. The rough structure can be identified by \eqref{eq:cs}, \eqref{eq:Deriv}, and the dual feasibility condition. It is worth mentioning that  the optimal support derived from our approach is often sparse compared to the number of moment constraints $n+1$, which is the dimension of the dual variable $\mathbf{z}$. Intuitively, this is because \eqref{eq:cs} and \eqref{eq:Deriv} combined together can be viewed as a linear system on $\bz$, and 
every supporting point $x_j$ with $1\le j \le K$ contributes one equation in \eqref{eq:cs} and 
one more equation if it appears in \eqref{eq:Deriv}.
$\mathbf{z}$ can therefore be determined by far fewer supporting points than its dimension, i.e., $K << n+1$. Our observation is also consistent with the previous results (see, e.g., \citealt{popescu2007robust,zuluaga2009third,he2010bounding,natarajan2018asymmetry,das2021heavy}), where the cardinality of the optimal support is often less than $n+1$.

Note that the information about the optimal solution that is provided by the rough structure in Step 1 is still limited. To proceed, in Step 2
we  define several scenarios  where the rough structure exhibits a more explicit expression, and then derive the 
possible (semi-)analytical form of the optimal supporting points for each scenario. This is achieved by treating  \eqref{eq:primal}, \eqref{eq:cs}, and \eqref{eq:Deriv} as a large nonlinear system of equations in $\bp$, $\bx$, and $\bz$ with a total of $K+k+n+1$ variables and an equal number of equations, and working intensively on the nonlinear system. Since this system is linear in $\bp$ and $\bz$ for any given $\bx$, we can represent $\bp$ and $\bz$ with $\bx$ by solving the linear system and further plugging the expression into the large nonlinear system to eliminate $\bz$ and $\bp$. As a result, we obtain a much smaller nonlinear system regarding {\it only} $\bx$ with $k$ variables and an equal number of equations.
For instance, in the proofs of Lemma \ref{lm:lg q} and Lemma \ref{lm:exp uv-1} we obtain two nonlinear equations involving only $\bx$ that are derived from the original nonlinear system with respect to $\bp$, $\bx$, and $\bz$. For more general cases, we can resort to 
classical methods such as the Gauss-Newton, trust region, and Levenberg-Marquardt methods (see, e.g., \citealt{more1978, Dennis1983, Kelley1995}) to solve such nonlinear systems.

The purpose of Step 3 is to compute the distribution according to the (semi-)analytical forms of the supporting points given in Step 2 and further verify the optimality of the constructed distribution. Since the (semi-)analytical forms are often derived by \eqref{eq:primal}, \eqref{eq:cs}, and \eqref{eq:Deriv}, Theorem \ref{tm: gpm} shows that it remains to verify the primal feasibility $\bp \ge 0$ and the dual feasibility $H(x;\bz) \geq 0$ for all $x \in \Omega$. 

We remark that the main effort in Step 3 is devoted to verifying the optimality of a given distribution, which plays a similar role to most analyses conducted in the literature to obtain the closed-form solution of a moment problem. However, an explanation of how to systematically provide such a distribution is largely missing in the extant research, while Step 1 together with Step 2 of our framework serve the purpose of finding one such distribution, thus complementing the extant literature. Moreover, any analytical solution of \eqref{Prob:GMP} must satisfy the large nonlinear system of equations jointly defined by \eqref{eq:primal}, \eqref{eq:cs}, and \eqref{eq:Deriv} in Step 2, as this is a necessary optimality condition according to Theorem \ref{tm: gpm}. Therefore, our approach can find any analytical optimal distribution as long as such a nonlinear system can be solved analytically, which has a good chance of occurring because the system has an equal number of variables and equations.

In what follows, we analyze three concrete moment problems with our unified framework: the $1$st and $t$-th moment problem, the moment problem with an upper partial moment objective and constraint, and the exponential moment problem. To 
analyze each problem,
we supplement our framework with detailed technical lemmas.
Specifically, the task of each step is accomplished by one or two of these lemmas, as summarized in Table \ref{Tab:correspondence} for the reader's convenience.  
\begin{table}[h]
\centering
\newcommand{\tabincell}[2]{\begin{tabular}{@{}#1@{}}#2\end{tabular}}
\begin{tabular}{|c|m{3.2cm}<{\centering}|m{3.2cm}<{\centering}|m{3.3cm}<{\centering}|}
\hline
 \diagbox{Framework}{Problems} &  $1$st and $t$-th moment problem  & upper partial moment problem & $1$st and exponential moment problem  \\
			\hline
			Step $1$ & Lemma \ref{lm:loc-1t} & Lemma \ref{lm: loc-121} & Lemma \ref{lm:loc-1et} \\  \hline
			Step $2$   & Lemma  \ref{lm:sm q}, \ref{lm:lg q} & Lemma \ref{lm: u-v-121}, \ref{lm: u-v-121-2} & Lemma \ref{lm:exp-0v}, \ref{lm:exp uv-1} \\  \hline
			Step $3$   & Lemma \ref{lm:sm q-b}, \ref{lm:lg q-b}  &Lemma \ref{lm: p-z-121}, \ref{lm: p-z-121-2}  & Lemma \ref{lm:exp 0v-2}, \ref{lm:exp uv-2} \\  \hline
		\end{tabular}
		\caption{Correspondence between technical lemmas and the steps in the framework}
		\label{Tab:correspondence}
\end{table}
\section{The $1$st and $t$-th moment problem in the newsvendor model}\label{Sec:1t}
\subsection{Problem formulation and main results}
In this section, we consider the moment problem that originated from a distributional robust newsvendor model proposed in \cite{das2021heavy} and demonstrate how to apply our framework to solve this problem. In particular, the problem is given as follows:
\begin{equation} \label{eq:P-1t}
\tag{\sf MP$_{1t}$}
\max_{F \in \F_{1t}}  \Ex_{F}[(X - q)_+],
\end{equation}
where $q$ is the order quantity in the newsvendor model and $F$ is the cumulative distribution function of random variable $X$. We also require that the distribution $F$ satisfies the following moment constraint \ given in \cite{das2021heavy}:
$$\F_{1t}=\left\{F \in \mathbb{M}(\R_+) : \int_{0}^\infty \mathrm{d}F(x)=1, \int_{0}^\infty x \mathrm{d}F(x)=M_1,\int_{0}^\infty x^t \mathrm{d}F(x)=M_t\right\},$$
where $M_1$ and $M_t$ are the $1$st and $t$-th moment parameters satisfying $M_t \geq M_1^t > 0$, with $t>1$. In fact, \cite{das2021heavy} showed that $\F_{1t}$ is capable of capturing more light-tailed ($t>2$) or heavy-tailed behavior ($t<2$) of the underlying distribution than the classical constraint in \cite{1958A}'s model. To analyze the moment problem
\eqref{eq:P-1t}, we also consider the following dual problem:
\begin{equation} \label{eq:D-1t}
\tag{\sf DMP$_{1t}$}
\begin{aligned} 
\min_z \,\,\, & &  z_0 +z_1M_1+z_t M_t & \\
\text{subject to} & & z_0+z_1 x+z_t x^t \geq (x - q)_+, & \,\,\,\text{for all }  x \in \mathbb{R}_+.
\end{aligned}
\end{equation}
Note that for any random variable $X$ with mean $M_1$, we can always perform a variable transformation by letting $\hat X=\frac{X}{M_1}$ such that 
$$\hat M_1=\Ex [\hat X]=1,\quad \hat M_t=\Ex [(\hat X)^t]=\frac{M_t}{M_1^t}, \quad \mbox{and}\quad  \hat q=\frac{q}{M_1},$$
and work on $\hat X$ instead.
Therefore, without loss of generality, we assume that $M_1=1$ in the remaining part of this section.
In this case, any single-point feasible distribution includes only the point $1$ in the support, and the $t$-th moment of this distribution is also equal to $1$. Thus, we further assume that $M_t>1$ to exclude the trivial cases of single-point feasible distributions. With our unified framework, we obtain the following main theorem of this section, which provides a characterization of the optimal value of problem \eqref{eq:P-1t}.
\begin{theorem} \label{tm:1t}
Suppose that $M_t>M_1=1$, with $t>1$.  Then the optimal value
of problem \eqref{eq:P-1t}
is given by
$$\max_{F \in \F_{1t}} \Ex_{F}[(X - q)_+]= \begin{cases}
      1-q M_t^{-\frac{1}{t-1}}, & \text{if } \;0<q\leq \frac{t-1}{t} M_t^{\frac{1}{t-1}}\\
      \frac{(v-q)(1-u)}{v-u}, & \text{if } \; q>\frac{t-1}{t} M_t^{\frac{1}{t-1}}\\
    \end{cases},     $$
where $u=\frac{tq}{t-1}\frac{v^{t-1}-M_t}{v^t-M_t}$ and $v \in \left(\max\left\{M_t^{\frac{1}{t-1}},q\right\},  \frac{t}{t-1}q\right)$ is any root of the following equation:
\begin{equation}\label{Eqn-Theta}
\Theta(y):=\frac{y^t-M_t}{y-1}\left(1-\frac{tq}{t-1}\frac{y^{t-1}-M_t}{y^t-M_t}\right)+\left(\frac{tq}{t-1}\frac{y^{t-1}-M_t}{y^t-M_t}\right)^t-M_t=0.
\end{equation}
\end{theorem}
When $0<q\le\frac{t-1}{t}M_t^{t-1}$, the optimal value stated in Theorem \ref{tm:1t} is consistent with the optimal value in \cite{das2021heavy}, Proposition 3.1. However, when $q>\frac{t-1}{t}M_t^{t-1}$, the problem \eqref{eq:P-1t} cannot be solved analytically. In particular, upper and lower bounds for the optimal value are derived in \cite{das2021heavy}, Propositions 3.3 and 3.4, while our Theorem \ref{tm:1t} provides a semi-closed form, which is probably the best one can hope for. One implication of the semi-closed form solution is that the optimal value
essentially depends on only one parameter that can be any root of a prespecified function, and one such root always exists and can be found efficiently by the bisection method presented in the next subsection. 
\subsection{The bisection algorithm}
In this subsection, we propose in Algorithm \ref{Alg: bm} a variant of the  bisection method that can find a root of the function $\Theta(\cdot)$ defined in \eqref{Eqn-Theta} in Theorem \ref{tm:1t} and thus solve the moment problem \eqref{eq:P-1t} when $q>\frac{t-1}{t}M_t^{t-1}$. 
\begin{algorithm}[H] 
\caption{Bisection method: {\bf BM}$(f,a,b, \epsilon)$} \label{Alg: bm}
\begin{algorithmic}[1]
\State {\bf input:} continuous function $f(x)$, interval $(a,b)$, tolerance $\epsilon>0$
 \If {$f(a)=0$ and $f'(a)=-sign(f(b))$} set $sign(f(a)) = -sign(f(b))$ \EndIf
\While{true}
    \State update $c = \frac{a+b}{2}$
    \If{$f(c)=0$ or $\frac{b-a}{2}\leq \epsilon$}  \text{Output($c$)}; \text{\bf Stop}\EndIf
    \If{$sign(f(b))=sign(f(c))$} set $b = c$ \Else{set $a = c$} \EndIf
\EndWhile
\end{algorithmic}
\end{algorithm}
Note that there are some subtle differences between Algorithm \ref{Alg: bm} and the standard bisection method. Specifically, Algorithm \ref{Alg: bm} aims to find a  root of a function in an open interval and allows the left starting point itself to be a root of the function, which is exactly the case for the function $\Theta(\cdot)$. In Figure \ref{Fg:1} we plot
one concrete instance of $\Theta(\cdot)$ to illustrate this situation.
In the following proposition, we formally state that Algorithm \ref{Alg: bm} is indeed able to find a root of $\Theta(\cdot)$. The proof is given in the appendix.
\begin{center}
\begin{figure} 
\centering
\includegraphics[width=10cm]{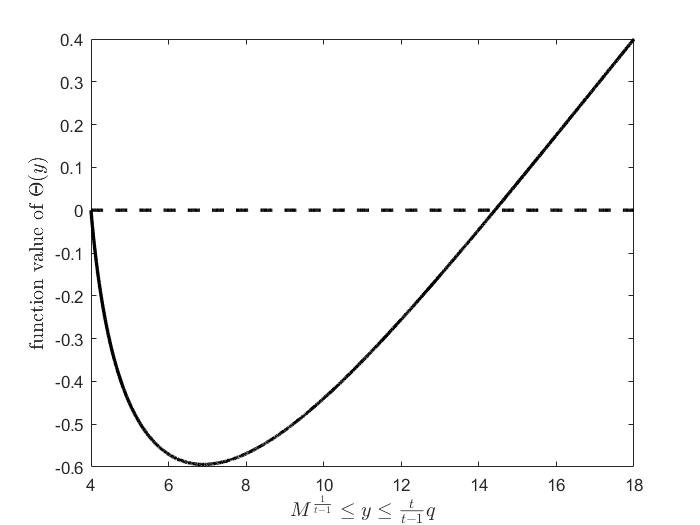}
\caption{{Sample plot of function $\Theta(\cdot)$ when the support is $\{u,v\}$ with $0<u<q<v$, $t=1$, $M_t = 2$, and $q=6$.}}
\label{Fg:1}
\end{figure}    
\end{center}
\begin{proposition}\label{prop:theta}
Suppose that the function $\Theta(\cdot)$ is defined as in \eqref{Eqn-Theta} and that $q>\frac{t-1}{t}M_t^{\frac{1}{t-1}}$. Then
\begin{enumerate}[i)]
	\item $\Theta\left(\frac{t}{t-1}q\right)>0$,
	\item $\Theta(q)<0$ when $q>M_t^\frac{1}{t-1}$,
	\item $\Theta\left(M^{\frac{1}{t-1}}\right)=0$ with $\Theta'\left(M^{\frac{1}{t-1}}\right)<0$ when $M_t^\frac{1}{t-1}\geq q$,
\end{enumerate}
and the bisection method given in Algorithm \ref{Alg: bm}
can find a root $v$  of equation \eqref{Eqn-Theta} in $ \left(\max\left\{M_t^{\frac{1}{t-1}},q\right\},  \frac{t}{t-1}q\right)$
 within $\log \left(\frac{q}{\epsilon (t-1)}\right)$ iterations for a given precision $\epsilon$.
\end{proposition}

To end this subsection, we remark that the iteration complexity of Algorithm \ref{Alg: bm} is independent of the value of $t$,
while the classical SDP approach works only if $t$ is a rational number, and its dimension depends on both the denominator and numerator of $t$.
\subsection{Proof of Theorem \ref{tm:1t}}
This subsection is dedicated to the proof of Theorem \ref{tm:1t}, which follows the three steps in our framework.
\subsubsection{Identify the  rough structure of the optimal support.}
To implement the first step of our framework, we shall show that the  optimal solution of \eqref{eq:P-1t} is a two-point distribution, and characterize how the two supporting points could be allocated. We first present the following lemma as preparation, relegating its
proof to the appendix.
\begin{lemma}\label{lemma:two-root}
Suppose that $H(x):=h_1(x) - h_2(x) \ge 0$ for all $x \in \R_+$, where $h_1(x)$ is a convex function with a strictly increasing derivative and $h_2(x)$ is a two-piece linear function such that
$$
h_2(x) = \left\{
\begin{array}{ll}
    a_1\,x + b_1, & \mbox{if}\;\; x \in [0, q) \\
    a_2\,x + b_2, & \mbox{if}\;\; x \in [q, +\infty)
\end{array}. \right.
$$
Then (i) there is at most one point $u \in [0, q)$ $\big{(}$or $u \in [0, q]$ if $h_2(x)$ is continuous at $q \big{)}$ such that $H(u) =0$, and (ii) there is at most one point $v \in [q, +\infty)$ such that $H(v) =0$. 
\end{lemma}
We are now ready to provide the rough structure of the optimal support in the following lemma.
\begin{lemma}[The rough structure of the optimal support]\label{lm:loc-1t} 
The optimal solution of \eqref{eq:P-1t} is a two-point distribution with support $\{u,v\}$ such that $0\leq u<q<v$.
\end{lemma}
\begin{proof} 
Since we are assuming that $M_t> M_1=1$, the optimal support cannot be a singleton. 
Let $\bz^*=(z^*_0,z^*_1,z^*_t)^T$ be the optimal solution of the dual problem \eqref{eq:D-1t}, and consider the dual constraint:
$$
H(x;\bz^*):=z^*_0+z^*_1 x+z^*_t x^t - (x - q)_+ \ge 0,\; \forall \; x \in \mathbb{R}_+.
$$ 
According to the complementary slackness condition \eqref{eq:cs}, the dual constraint must be tight at the points in the support of the primal optimal distribution. Therefore, in what follows we shall identify the optimal support by searching the points where the dual constraint could be tight.
We first note that $z^*_t \ge 0$, since otherwise $H(x;\bz^*)<0$ for  sufficiently large $x$ as $t>1$, which contradicts the feasibility of $\bz^*$. Now we consider the easy case of $z^*_t=0$, where the dual constraint reduces to
$$
(z^*_0+q)+ (z^*_1-1)\,x=z^*_0+ z^*_1x- (x - q) \ge 0,\; \forall \; x \ge q,\quad\mbox{and}\quad z^*_0+ z^*_1x \ge 0, \; \forall \; 0 \le x \le q,
$$
which gives us $z^*_1 \ge 1$ and $z^*_0 \ge 0$. As $\bz^*$ minimizes $z_0 + M_1 z_1$ in the objective and $M_1>0$, we obtain  $z^*_0 = 0$ and $z^*_1 = 1$. Consequently, $H(x;\bz^*) = x - (x-q)_+ = 0$ only when $x =0$. That is, the optimal solution of \eqref{eq:P-1t} is a single-point distribution with support $\{0\}$, contradicting the first sentence of the proof. 
Thus, we must have $z^*_t>0$, and the function $z^*_0+z^*_1 x+z^*_t x^t$ is convex in $x$ and has a strictly increasing derivative. Moreover, since $(x-q)_+$ is a two-piece linear function and is continuous at the break point $q$, Lemma \ref{lemma:two-root} implies that there are at most one point $u \in [0, q]$ and at most one point $v \in [q, +\infty)$ such that $H(u;\bz^*)=H(v;\bz^*) =0$. In other words, the optimal distribution of \eqref{eq:P-1t} has two support points, one in the interval $[0, q]$ and one in the interval $[q, +\infty)$. Moreover, we have $u\neq q \neq v$, for otherwise $u$ and $v$ are both included in either $[0,q]$ or $[q, +\infty)$ or the optimal distribution is supported by a single point, which leads to a contradiction. Therefore, the two support points satisfy $0 \le u < q < v$.
\end{proof}

To make the structure of the optimal support in Lemma \ref{lm:loc-1t} more explicit,  
 we shall continue our analysis with two scenarios, one in which the supporting point $u$ equals $0$ and one in which it is strictly greater than $0$.  
\subsubsection{Analysis under the support $\{0,v\}$ with $v>q$.}
We start with the boundary case $u=0$ and provide the analytical form of the optimal support as follows.
\begin{lemma}[Analytical form of the supporting points]\label{lm:sm q}
Suppose that $q>0$, $M_t>M_1=1$ with $t>1$, and the optimal solution of \eqref{eq:P-1t} is a two-point distribution with support $\{0,v\}$ such that $v>q$. Then
we must have $v=M_t^{\frac{1}{t-1}}$
and $q\leq \frac{t-1}{t} M_t^{\frac{1}{t-1}}$.
\end{lemma}
\begin{proof}
Suppose that the optimal distribution of (\ref{eq:P-1t}) is supported by $\{0,v\}$ with probabilities $p_1^*$ and $p_2^*$, respectively, and that $\bz^*=(z^*_0,z^*_1,z^*_t)^T$ is an optimal solution of the dual problem \eqref{eq:D-1t}.
Let 
\begin{equation}\label{Func:H-1t-u-v}
{H}(x;\bz^*)=z_0^*+z_1^*x +z_t^*x^t-(x-q)_+.  
\end{equation}
In this case, condition \eqref{eq:Deriv} gives ${H}'(v;\bz^*)= z_1^* + t v^{t-1}  z_t^* -1 =0$. According to Theorem \ref{tm: gpm}, $\bz^*$ and $\bp^* = (p_1^*, p_2^*)^T$ satisfy conditions \eqref{eq:primal}, \eqref{eq:cs}, and \eqref{eq:Deriv}, which are equivalent to
\begin{equation} \label{Cond:newsvendor-1t}
\left[
\begin{array}{cc}
1 & 1   \\
0 & v  \\
0 & v^{t}\\ 
\end{array}
\right] \left[  
\begin{array}{c}
p_1^*\\
p_2^* \\
\end{array}
\right] =  \left[  
\begin{array}{c}
1\\
1 \\
M_t \\
\end{array}
\right]
\,\,\, \text{and  } \,\,\,
\left[
\begin{array}{ccc}
1 & 0 & 0  \\
1 & v & v^t \\
0 & 1 &t v^{t-1} 
\end{array}
\right] \left[  
\begin{array}{c}
z^*_0\\
z^*_1 \\
z^*_t\\
\end{array}
\right] =  \left[  
\begin{array}{c}
0\\
v-q \\
1 \\
\end{array}
\right].
\end{equation}
Solving the above equations gives 
\begin{equation}\label{Sol:1t}   
v= M_t^{\frac{1}{t-1}}>1, \quad\;
\left[  
\begin{array}{c}
p_1^* \\
p_2^* \\
\end{array}
\right]= \left[  
\begin{array}{c}
1-M_t^{-\frac{1}{t-1}}\\
M_t^{-\frac{1}{t-1}}\\
\end{array}
\right],
\,\,\, \text{and } \,\,\, 
\left[  
\begin{array}{c}
z^*_0\\
z^*_1 \\
z^*_t\\
\end{array}
\right]= \left[  
\begin{array}{c}
0\\
1-\frac{tq}{t-1} M_t^{-1/(t-1)} \\
\frac{q}{t-1} M_t^{-t/(t-1)}
\end{array}
\right]. 
\end{equation}
In addition, the second system of equations in \eqref{Cond:newsvendor-1t} indicates that $z^*_0=H(0;\bz^*)=0$. By invoking
the dual feasibility of $\bz^*$, we have
$$
x(z^*_1 +z^*_t x^{t-1})=z^*_1 x+z^*_t x^t=z^*_0+z^*_1 x+z^*_t x^t = H(x;\bz^*)>0,\;\forall\;0<x<q.
$$
That is, $z^*_1 +z^*_t x^{t-1}\ge 0$ for arbitrarily small $x$. Therefore, $z^*_1 \ge 0$,
which combined with \eqref{Sol:1t} guarantees that $q\leq \frac{t-1}{t} M_t^{\frac{1}{t-1}}$. 
\end{proof}

Given the analytical form of the optimal support, we can solve and verify the optimal distribution for this case.
\begin{lemma}[Analytical form of the optimal distribution] \label{lm:sm q-b}
Suppose that $M_t>M_1=1$, with $t>1$ and $\frac{t-1}{t} M_t^{\frac{1}{t-1}} \ge q>0$. Then the optimal distribution for problem \eqref{eq:P-1t} can be characterized as
\begin{equation}
\label{opt-disn-1t-a}
X^*= \begin{cases}
      0, & w.p. \;1-  M_t^{-\frac{1}{t-1}}\\
       M_t^{\frac{1}{t-1}}, & w.p. \; M_t^{-\frac{1}{t-1}}\\
    \end{cases},  \quad \mbox{with the optimal value }\;   \max_{F \in \F_{1t}}  \Ex_{F}[(X - q)_+]=1-q M_t^{-\frac{1}{t-1}}.
\end{equation}
\end{lemma}
\begin{proof} We first construct 
$\bp^*=(p^*_1, p^*_2)^T$ and $\bz^*=(z^*_0, z^*_1, z^*_t)^T$ in accordance with \eqref{Sol:1t}, where $p^*_1$ and $p^*_2$ correspond to the probabilities associated with $0$ and $ M_t^{\frac{1}{t-1}}$ in the support of the distribution described in \eqref{opt-disn-1t-a}. 
We shall show that $\bp^*$ and $\bz^*$ are the optimal solutions for the primal problem \eqref{eq:P-1t} and the dual problem \eqref{eq:D-1t}, respectively.
Recalling that $M_t>1$, $ p_1^*, p_2^* \in (0,1)$, and thus $\bp^*$ is a primal feasible solution. 
Next, we verify the dual feasibility of $z^*$, i.e., ${H}(x;\bz^*)\ge 0$ for any $x \ge 0$ with ${H}(x;\bz^*)$ defined as in \eqref{Func:H-1t-u-v}.
According to \eqref{Sol:1t} and $\frac{t-1}{t} M_t^{\frac{1}{t-1}} \ge q$, we have $z_0^*=0$, $z_1^* \geq 0$, and $z_t^* > 0$, and hence $H(x;\bz^*) \ge 0$ for any $x \in [0,q]$.  When $x \ge q$, 
$H(x;\bz^*) $ is convex, as in this case, $${H}''(x;\bz^*)=t(t-1)z_t^*x^{t-2}=t q M_t^{-t/(t-1)}x^{t-2}>0.$$
In addition, we have $H\left(M_t^{1/(t-1)};\bz^*\right)=H(v;\bz^*)=H'(v;\bz^*)=0$, as $z^*$ is the solution to the second set of equations in \eqref{Cond:newsvendor-1t}, and $x=M_t^{1/(t-1)}$ is the unique root of $H'(x;\bz^*)= z_1^* + t z_t^*x^{t-1}-1=0$.
Then the global minimizer of $H(x;\bz^*)$ on $[q,+\infty)$ is taken either at the point $M_t^{1/(t-1)}$ or at the end point $q$, where we already have shown that $H(q;\bz^*) \ge 0$.
Therefore, 
$H(x;\bz^*) \ge \min\{H(q;\bz^*), H(M_t^{1/(t-1)};\bz^*) \} = 0 $ holds when $x \ge q$, and $\bz^*$ is a feasible solution for the dual problem \eqref{eq:D-1t}. 
In addition, $(\bp^*, \bz^*)$ also satisfies the complementary slackness condition due to the second set of equations in \eqref{Cond:newsvendor-1t}, and thus we conclude that $(\bp^*, \bz^*)$ is an optimal primal-dual solution pair such that the distribution defined in \eqref{opt-disn-1t-a} is optimal for \eqref{eq:P-1t}. 
\end{proof}
\subsubsection{Analysis under the support $\{u,v\}$ with $0< u < q < v$.}
We now consider the case where $u>0$ and provide the semi-analytical form of the optimal support.
\begin{lemma}[Semi-analytical form of the supporting points] \label{lm:lg q}
Suppose that $M_t>M_1=1$ with $t>1$ and $q>0$, and that the optimal solution of \eqref{eq:P-1t} is a two-point distribution with support $\{u,v\}$ such that $0<u<q<v$. Then we must have $u=\frac{tq}{t-1}\frac{v^{t-1}-M_t}{v^t-M_t}$, and $v \in \left(\max\left\{M_t^{\frac{1}{t-1}},q\right\},  \frac{t}{t-1}q\right)$ is a root of equation \eqref{Eqn-Theta}. Moreover, in this case we also have
$q> \frac{t-1}{t} M_t^{\frac{1}{t-1}}$.
\end{lemma}
\begin{proof}
Suppose that the optimal distribution of (\ref{eq:P-1t}) is supported by $\{u,v\}$ with probabilities $p_1^*$ and $p_2^*$, respectively,
and that $\bz^*=(z^*_0,z^*_1,z^*_t)^T$ is an optimal solution of the dual problem (\ref{eq:D-1t}).
Recalling that ${H}(x;\bz^*)$ is defined in \eqref{Func:H-1t-u-v}
and according to Theorem \ref{tm: gpm}, we have ${H}'(u;\bz^*)= z_1^* + t u^{t-1}  z_t^* =0 $ and ${H}'(v;\bz^*)= z_1^* + t v^{t-1}  z_t^*-1 =0 $.
Therefore, $\bz^*$ and $\bp^* = (p_1^*, p_2^*)^T$ satisfy the conditions \eqref{eq:primal}, \eqref{eq:cs}, and \eqref{eq:Deriv}, which are equivalent to
\begin{equation}  \label{eq:1-t large}
\left[
\begin{array}{cc}
1 & 1  \\
u & v  \\
u^t & v^t \\
\end{array}
\right] \left[  
\begin{array}{c}
p_1^*\\
p_2^* \\
\end{array}
\right]= \left[  
\begin{array}{c}
1\\
1 \\
M_t \\
\end{array}
\right]
\,\,\,\text{and } \,\,\,
\left[
\begin{array}{ccc}
1 & u & u^t  \\
1 & v & v^t \\
0 & 1 &t u^{t-1} \\
0 & 1 & t v^{t-1} \\
\end{array}
\right] \left[  
\begin{array}{c}
z^*_0\\
z^*_1 \\
z^*_t\\
\end{array}
\right] = \left[  
\begin{array}{c}
0\\
v-q \\
0 \\
1 \\
\end{array}
\right].
\end{equation}
The first system of equations in \eqref{eq:1-t large} has a solution if and only if
\begin{eqnarray*}
	0 = \left|
	\begin{array}{ccc}
1 & 1  & 1\\
u & v  & 1\\
u^t & v^t & M_t\\
	\end{array}
	\right|  \overset{\rm{row2-row1,  row3- M_t \cdot row1 }}= \left|
	\begin{array}{ccc}
1 & 1  & 1\\
u-1 & v-1  & 0\\
u^t-M_t & v^t-M_t & 0\\
	\end{array}
	\right|= \left|
	\begin{array}{cc}
u-1 & v-1  \\
u^t-M_t & v^t-M_t \\
	\end{array}
	\right|. \\
\end{eqnarray*}
That is,
\begin{align}\label{eq:u-v-1t}
\frac{M_t-u^t}{1-u}=\frac{M_t-v^t}{1-v}=\frac{v^t-u^t}{v-u}.
\end{align}
By a similar argument, the second set of equations in \eqref{eq:1-t large} has a solution if and only if
\begin{eqnarray}
	  0&=&\left|
	\begin{array}{cccc}
		1 & u & u^t  & 0\\
		1 & v & v^t  & v-q\\
		0 & 1 & t u^{t-1}  & 0\\
		0 & 1 & t v^{t-1} & 1\\
	\end{array}
	\right|  \overset{\rm{row2-(v-q)\cdot row4}}= \left|
	\begin{array}{cccc}
1 & u & u^t  & 0\\
1 & q & v^t-(v-q)t v^{t-1}  & 0\\
0 & 1 & t u^{t-1}  & 0\\
0 & 1 & t v^{t-1} & 1\\
	\end{array}
	\right|
	= \left|
	\begin{array}{ccc}
1 & u & u^t  \\
1 & q & v^t-(v-q)t v^{t-1}  \\
0 & 1 & t u^{t-1}  \\
	\end{array}
	\right| \nonumber\\
	& \overset{\rm{row2- row1}}= &\left|
	\begin{array}{ccc}
1 & u & u^t  \\
0 & q-u & v^t-u^t-(v-q)t v^{t-1}  \\
0 & 1 & t u^{t-1}  \\
	\end{array}
	\right| 
	 =
	\left|
	\begin{array}{cc}
 q-u & v^t-u^t-(v-q)t v^{t-1}  \\
 1 & t u^{t-1}  \\
	\end{array}
	\right|, \nonumber
\end{eqnarray}
 which is precisely $(q-u)t u^{t-1} =v^t-u^t-(v-q)t v^{t-1}$, or equivalently, 
 \begin{equation}
  \frac{v^t-u^t}{v^{t-1}-u^{t-1}}=\frac{tq}{t-1}. \label{Det=0-1t}   
 \end{equation}
Consequently, we have 
$$1-\frac{(t-1)u}{tq}=1-\frac{v^{t-1}u-u^t}{v^t-u^t}=v^{t-1}\frac{v-u}{v^t-u^t}=v^{t-1}\frac{v-1}{v^t-M_t},$$
where the last equality is due to \eqref{eq:u-v-1t}.
Therefore, we can represent
$u$ as $\frac{t q}{t-1}\frac{v^{t-1}-M_t}{v^t-M_t}$, and substitute $u$ in 
$\frac{v^t-M_t}{v-1}(1-u)+u^t-M_t=0$, which is a reformulation of \eqref{eq:u-v-1t}. Consequently, $v$ is a root of $\Theta(\cdot)$. In addition, we have further estimations of the range of $v$. In particular, due to our assumption that $v>u>0$ and the validity of the $t$-th moment constraint, we have $u=\frac{tq}{t-1}\frac{v^{t-1}-M_t}{v^t-M_t}>0$ and $v^t>p_1^*u^t+p_2^* v^t=M_t$. Combining those two inequalities yields 
$v^{t-1}-M_t > 0$, or equivalently, $v > M_t^{\frac{1}{t-1}}$. Equation
\eqref{Det=0-1t} also implies that $\frac{tq}{t-1}= \frac{v^t-u^t}{v^{t-1}-u^{t-1}}=v+u^{t-1}\frac{v-u}{v^{t-1}-u^{t-1}}>v$, and thus we conclude that $v \in \left(\max\left\{M_t^{\frac{1}{t-1}},q\right\},  \frac{t}{t-1}q\right)$. As the upper bound $\frac{tq}{t-1}$ must exceed the lower bound $M^{\frac{1}{t-1}}$, we have $q>\frac{t-1}{t}M^{\frac{1}{t-1}}$.
\end{proof}.

Finally, the semi-analytical form of the optimal support enables us to identify the optimal distribution as follows.
\begin{lemma}[Semi-analytical form of  the optimal distribution]\label{lm:lg q-b}
Suppose that $M_t>M_1=1$, with $t>1$, and $q> \frac{t-1}{t} M_t^{\frac{1}{t-1}}$. Let 
$u=\frac{tq}{t-1}\frac{v^{t-1}-M_t}{v^t-M_t}$ and let $v \in \left(\max\left\{M_t^{\frac{1}{t-1}},q\right\},  \frac{t}{t-1}q\right)$ be any root of equation \eqref{Eqn-Theta}. Then we have $0<u<q<v$, and the optimal distribution of problem \eqref{eq:P-1t} can be  characterized  as
\begin{equation}\label{opt-disn-1t-b}
   X^*= \begin{cases}
      u, & w.p. \;\frac{v-1}{v-u}\\
      v, & w.p.  \;\frac{1-u}{v-u}\\
    \end{cases},\quad \mbox{with optimal value} \;   \max_{F \in \F_{1t}}  \Ex_{F}[(X - q)_+]=\frac{(v-q)(1-u)}{v-u}. 
\end{equation}
\end{lemma}
\begin{proof}
Suppose that $v \in \left(\max\left\{M_t^{\frac{1}{t-1}},q\right\},  \frac{t}{t-1}q\right) $ is any root of $\Theta(\cdot)=0$ and construct $u=\frac{t q}{t-1}\frac{v^{t-1}-M_t}{v^t-M_t}>0$. We first want to show that $q>u =\frac{t q}{t-1}\frac{v^{t-1}-M_t}{v^t-M_t}$, or equivalently, that the function $G_1(y):=t y^{t-1}-(t-1)y^t - M_t<0$ at point $v$. Since $G'_1(y) = t(t-1)y^{t-2}(1-y)<0$ when $y>1$, $G_1(y)$ is a strictly decreasing function on $[1,+\infty)$, which combined with $v > M_t^{\frac{1}{t-1}}$ implies 
$$G_1(v) < G_1(M_t^{\frac{1}{t-1}})=t\left(M_t^{\frac{1}{t-1}} \right)^{t-1}-(t-1)\left(M_t^{\frac{1}{t-1}} \right)^{t}-M_t=(t-1)\left(M_t - M_t^{\frac{t}{t-1}}\right)<0.$$
Therefore, $0< u < q < v$ holds and the construction
\begin{equation*} 
    \left[  
\begin{array}{c}
p_1^* \\
p_2^* \\
\end{array}
\right]= \left[  
\begin{array}{c}
 \frac{v-1}{v-u} \\
\frac{1-u}{v-u} \\
\end{array}
\right] 
\,\,\, \text{and } \,\,\,
\left[  
\begin{array}{c}
z^*_0\\
z^*_1 \\
z^*_t\\
\end{array}
\right] =  \left[  
\begin{array}{c}
\frac{(t-1)u^t}{t(v^{t-1}-u^{t-1})}\\
-\frac{u^{t-1}}{v^{t-1}-u^{t-1}} \\
\frac{1}{t(v^{t-1}-u^{t-1})} \\
\end{array}
\right]
\end{equation*}
is well defined, and $\bp^*=(p^*_1, p^*_2)^T$ and $\bz^*=(z^*_0, z^*_1, z^*_t)^T$ are the solutions of the two linear equations in \eqref{eq:D-1t}. Therefore, 
to confirm the primal feasibility, all that remains is to show that $0< p^*_1,p^*_2<1$, or equivalently, $u < 1 < v$, due to the construction of $p^*$. We first observe that $v>M_t^{\frac{1}{t-1}}>1$, as $v \in \left(\max\left\{M_t^{\frac{1}{t-1}},q\right\},  \frac{t}{t-1}q \right)$. Next we consider the function $G_2(y)=\frac{y^t-M_t}{y-1}$, and a straightforward computation shows that
$$
G'_2(y)= \frac{(t-1)y^t -ty^{t-1} +M_t }{(y-1)^2}\;\mbox{and}\; \left( (t-1)y^t -ty^{t-1} +M_t \right)'=t(t-1)y^{t-2}(y-1)>0\;\mbox{when}\;y>1.
$$
Therefore,  as long as $y>1$, $G'_2(y)>0$, and thus $G_2(y)$ is a strictly increasing function. Recall that we have already proved that $v>\max\{1,u\}$ and $G_2(v)=\frac{v^t-M_t}{v-1}=\frac{u^t-M_t}{u-1}=G_2(u)$ holds by \eqref{eq:u-v-1t}. Then we must have $u<1$, as desired, for $G_2(y)$ is strictly increasing on $(1,+\infty)$ and cannot have the same value at $u$ and $v$.
Next, we verify the dual feasibility of $\bz^*$.
Recalling that $H(x;\bz^*)$ is defined as in \eqref{Func:H-1t-u-v}, the second equation in \eqref{eq:1-t large} indicates that
\begin{equation}\label{H-u-v-0}
 {H}'(u;\bz^*)={H}'(v;\bz^*)=0\; \;\,\mbox{and}\;\, {H}(v;\bz^*)={H}(u;\bz^*)=0.
\end{equation}
Moreover, 
since we have shown that $u<1<v$, $z_t^*=\frac{1}{t(v^{t-1}-u^{t-1})}>0$ and ${H}''(x;\bz^*)=t(t-1)z_t^*x^{t-2}\ge 0$ when $x\in [0,q) \bigcup (q,\infty)$. Therefore, ${H}(x;\bz^*)$ is convex on $[0,q)$ and $(q,\infty)$, which combined with
$u <  q < v$ and \eqref{H-u-v-0} implies that
$$
{H}(x;\bz^*) \ge \left\{\begin{array}{cc}
      {H}(v;\bz^*)=0,& \mbox{when}\quad x \in (q,\infty) \\
     {H}(u;\bz^*)=0, & \mbox{when}\quad x \in [0,q)
\end{array}
\right. .
$$
Since ${H}(x;\bz^*)$ is continuous at $q$, we conclude that ${H}(x;\bz^*) \ge 0$ for all $x \ge 0$ and $\bz^*$ is a dual feasible solution. Observe that 
the complementary slackness condition is already implied by 
${H}(u;\bz^*)={H}(v;\bz^*)=0$, and hence
$(\bp^*, \bz^*)$ is indeed an optimal primal-dual solution such that the distribution defined in \eqref{opt-disn-1t-b} is optimal for \eqref{eq:P-1t}. 
\end{proof}
\subsubsection{Proof of Theorem \ref{tm:1t}.}
Step $2$ and Step $3$ of our framework have been accomplished under the two scenarios, and accordingly, Theorem \ref{tm:1t} readily follows from Lemma \ref{lm:sm q-b} and Lemma \ref{lm:lg q-b}.
\section{Minimizing the upper partial moment with the 1st-order upper partial moment constraint}\label{sec:P-121}
\subsection{Problem formulation and the main results}
In this section, we consider the following moment constraint defined by the 1st-order upper partial moment:
$$\F_{121_+}=\left\{F \in \mathbb{M}(\R_+) : \int_{0}^\infty \mathrm{d}F(x)=1, \int_{0}^\infty x \mathrm{d}F(x)=M_1,\int_{0}^\infty x^2 \mathrm{d}F(x)=M_2,\int_{0}^\infty (x-q)_+\mathrm{d}F(x)=M_+\right\},$$
where $M_1, M_2,$ and $M_+$ are the $1$st moment, $2$nd moment, and $1$st-order upper partial moment, respectively. Given such a constraint, we want to minimize the $2$nd-order upper partial moment given as follows:
$$
\mathbb{V}ar[(X - q)^2_+]= \mathop{\mathbb{E}} [(X - q)^2_+] -M_+^2,\;\mbox{for given}\; q>0.
$$
This problem appears in (EC.4) of \cite{han2014risk}, and it is the key to deriving the analytical results for the risk-averse and ambiguity-averse newsvendor problem in Theorem 1 therein. Adopting the normalization idea given in \cite{han2014risk}, we assume that $q=1$ and $M_2=\gamma M_1^2$ for some $\gamma>0$ without loss of generality. Therefore, the corresponding moment problem is given by
\begin{equation} \label{eq:P-121}
\tag{\sf MP$_{121_{+}}$}
\begin{aligned}
\min_{F \in \F_{121_+}} \,\,\,\, &  \int_{0}^\infty (x-1)_+^2 \mathrm{d}F(x) -M_+^2& \\
\end{aligned}
\end{equation}
and the associated dual formulation is given by
\begin{equation} \label{eq:D-121}
\tag{\sf DMP$_{121_{+}}$}
\begin{aligned} 
\max_\bz \,\,\, & &  z_0 +z_1M_1+z_2 \gamma M_1^2+z_3M_+ -M_+^2& \\
\text{subject to} & & z_0+z_1 x+z_2 x^2+z_3(x-1)_+ \leq(x - 1)^2_+, & \,\,\,\text{for all }  x \in \mathbb{R}_+.		
\end{aligned}
\end{equation}
By applying our unified framework, we get the same 
optimal value for problem \eqref{eq:P-121} as that provided in \cite{han2014risk}, in the following theorem. Moreover, our framework is able to discover a degenerate case with infinitely many optimal supporting points (Lemma \ref{lm: u-v-121-2}) and to identify infinitely many optimal solutions (Lemma \ref{lm: p-z-121-2}) that are not mentioned in \cite{han2014risk}.
\begin{theorem} \label{tm:121}
Suppose that $\gamma>1, M_+>0$, and  problem \eqref{eq:P-121} is feasible.
 Then the optimal value of \eqref{eq:P-121} is given by $\inf_{F \in \F_{121_+}}  \int_{0}^\infty (x-1)_+^2 \mathrm{d}F(x) -M_+^2$
 \begin{eqnarray*}
&=& \begin{cases}
     M_1(\gamma M_1 -1)-M_+-M_+^2, & \text{if } \; M_1 > \frac{1}{\gamma}+M_+\\
     \frac{1}{2}\left(2M_+(M_1-1)+M_1\left(M_1(\gamma-1)-\kappa_{M_1,\gamma,M_+}\right)\right), & \text{if } \; M_1 \leq \frac{1}{\gamma}+M_+ \\
    \end{cases},    
    \end{eqnarray*}
where 
\begin{equation}\label{eqn:kappa-121}
\kappa_{M_1,\gamma,M_+}=\sqrt{(\gamma-1)[(\gamma-1)M_1^2+4M_+(M_1-1)-4M_+^2]}.
\end{equation}
\end{theorem}
In the rest of this section, we shall only highlight the major flow of
our analysis. Specifically, we present the contents of the key technical lemmas that correspond to the three steps in our framework, but relegate the
proofs to the appendix.
\subsection{Sketch of Proof for Theorem \ref{tm:121}}
We first provide the rough structure of the optimal support that realizes the goal of the first step in our framework.
\begin{lemma}[The rough structure of the optimal support] \label{lm: loc-121}
    Suppose that $\gamma>1$, $M_+>0$, and $\bz^*=(z^*_0,z^*_1,z^*_2,z^*_3)^T$ is an optimal solution of the dual problem \eqref{eq:D-121}. Let $S$ be the support set for the optimal distribution of \eqref{eq:P-121}. Then we have $z_2^* \le 1$ and, moreover,
    \begin{enumerate}[i)]
        \item $S=\{u,v\}$ with $0 \leq u<1<v$ if $z^*_2<0$;
        \item $S=\{0,v\}$ with $v>1$ if $0 \leq z^*_2<1$;
        \item $S \subseteq \{0\} \bigcup [1,\infty)$ if $z^*_2=1$.  
    \end{enumerate}
\end{lemma}
Note that the support obtained in part ii) of Lemma \ref{lm: loc-121} is a special case of i). Our discussion in the remainder of this section will therefore consider only  the expressions of the support described in parts i) and iii).
\subsubsection{Results under the support $S=\{u,v\}$ with $0 \le u<1<v$.} \label{subsec: u-v-121}
We start with the case where the optimal support includes only two points and provide the following analytical form of the support.
\begin{lemma}[Analytical form of the support $S=\{u,v\}$]\label{lm: u-v-121}
	Suppose that $\gamma>1$, $M_+>0$, and the optimal solution of  \eqref{eq:P-121} is a two-point distribution with support $\{u,v\}$ {such that} $0 \leq u<1<v$. Then we must have  $M_1 \leq  M_+ +\frac{1}{\gamma}$, $u = M_1\left(1-\frac{(\gamma-1)M_1+\kappa_{M_1,\gamma,M_+}}{2(1-M_1+M_+)}\right)$, and $v = M_1\left(1+\frac{(\gamma-1)M_1-\kappa_{M_1,\gamma,M_+}}{2M_+}\right)$, {where $\kappa_{M_1,\gamma,M_+}$ is} defined as in \eqref{eqn:kappa-121}.
\end{lemma}
Given the analytical form in Lemma \ref{lm: u-v-121}, we can solve and verify the optimal distribution in this case.
\begin{lemma}[Analytical form of the optimal distribution with support $S=\{u,v\}$] \label{lm: p-z-121}
Suppose that $\gamma>1, M_+>0$, $M_1 \leq \frac{1}{\gamma}+M_+$, and  problem \eqref{eq:P-121} is feasible. Then the optimal distribution for problem \eqref{eq:P-121} can be characterized as
\begin{equation}\label{OptDis:uv-121}
X^*= \begin{cases}
M_1\left(1-\frac{(\gamma-1)M_1+\kappa_{M_1,\gamma,M_+}}{2(1-M_1+M_+)}\right), & w.p.  \;1-\frac{M_1\left(2M_+(\gamma-1)M_1+\kappa_{M_1,\gamma,M_+}\right)-2M_+}{2(1-2M_1+\gamma M_1^2)}\\
M_1\left(1+\frac{(\gamma-1)M_1-\kappa_{M_1,\gamma,M_+}}{2M_+}\right), & w.p. \;\frac{M_1\left(2M_+(\gamma-1)M_1+\kappa_{M_1,\gamma,M_+}\right)-2M_+}{2(1-2M_1+\gamma M_1^2)}\\
\end{cases}     
\end{equation}
with the optimal value 
$$\inf_{F \in \F_{121_+}}  \mathbb{V}ar[(X - q)_+]=\frac{1}{2}\Big(2M_+(M_1-1)+M_1\left((\gamma-1)M_1-\kappa_{M_1,\gamma,M_+}\right)\Big)-M_+^2,$$
where $\kappa_{M_1,\gamma,M_+}$ is defined as in \eqref{eqn:kappa-121}.
\end{lemma}
\subsubsection{Results under the support $S\subseteq \{0\}\bigcup [1,+\infty)$.}
Now we consider the degenerate case in which the optimal support includes infinitely many points and provide the following analytical form for the three-point optimal support.
\begin{lemma}[Analytical form of the support $S\subseteq \{0\}\bigcup [1,+\infty)$]\label{lm: u-v-121-2}
	Suppose that $\gamma>1$, $M_+>0$, and the optimal solution of  \eqref{eq:P-121} has the support $S\subseteq \{0\}\bigcup [1,+\infty)$. Then we have $M_1> \frac{1}{\gamma}+ M_+$, and any feasible support $\{0,v_1,v_2,...,v_d\}$ with $v_i \ge 1$ for all $1 \le i \le d$ has the same objective value and thus is an optimal support. Moreover, any three-point optimal support has the form  $\left\{0,v_1,\frac{M_1v_1 -\gamma M_1^2}{(M_1-M_+)v_1 -M_1}\right\}$, with
	$ v_1 \ge \max\left\{1, \frac{\gamma M_1^2-M_1}{M_+}\right\}$.
\end{lemma}
The analytical form in the preceding lemma enables us to find the optimal distribution in the following lemma.
\begin{lemma}[Analytical form of the optimal distribution with support $S\subseteq \{0\}\bigcup [1,+\infty)$] \label{lm: p-z-121-2}
	Suppose that $\gamma>1$,  $M_+>0$, $ M_1 >  \frac{1}{\gamma}+M_+$, and problem \eqref{eq:P-121} is feasible. Then an optimal three-point distribution of problem \eqref{eq:P-121}  can be characterized as 
	\begin{equation}\label{OptDis:0v2-121}
	X^*= \begin{cases}
	0, & w.p.   \; 1-M_1+M_+,\\
	v_1, & w.p. \; \frac{M_1^2(\gamma M_1-\gamma M_+ -1)}{\gamma M_1^2 -2M_1 v_1+M_1v_1^2-M_+v_1^2},\\
	\frac{M_1v_1-\gamma M_1^2}{(M_1-M_+)v_1-M_1}, & w.p. \; \frac{(M_1v_1-M_+v_1-M_1)^2}{\gamma M_1^2 -2M_1 v_1+M_1v_1^2-M_+v_1^2}
	\end{cases}       
	\end{equation} 
	for any $v_1 \ge \max\left\{1, \frac{\gamma M_1^2-M_1}{M_+}\right\}$,
	and the optimal value is given by
	$\inf_{F \in \F_{121_+}}  \mathbb{V}ar[(X - q)_+]=M_1(\gamma M_1-1)-M_+-M_+^2$.
\end{lemma}
\subsubsection{Proof of Theorem \ref{tm:121}.}
The goals in Step $2$ and Step $3$ of our framework have been achieved under the two scenarios, and
Theorem \ref{tm:121} readily follows from Lemmas \ref{lm: p-z-121} and \ref{lm: p-z-121-2}.
\section{The $1$st and exponential moment problem in the newsvendor model}\label{sec:P-1et}
\subsection{Problem formulation and main results}
In this section, we further demonstrate the capability of our framework by solving a problem with
an exponential moment constraint that cannot be handled by the SDP or by the relative entropy approach in \cite{das2021heavy}. Moreover, as it will be shown in a numerical experiment, the model with  this moment information is prone to capture the features of the light tailed distribution. Specifically,
the moment constraint that we consider is given by
$$\F_{1e}=\left\{F \in \mathbb{M}(\R_+) : \int_{0}^\infty \mathrm{d}F(x)=1, \int_{0}^\infty x \mathrm{d}F(x)=M_1,\int_{0}^\infty e^{tx} \mathrm{d}F(x)=M_e\right\},$$
where $t>0$ is fixed, and $M_1$ and $M_e$ are the $1$st and exponential moment parameters satisfying $M_e \geq e^{tM_1}$ such that $\F_{1e}$ is not empty. Moreover, as $M_e =e^{tM_1}$ only holds for single-point distributions, we assume that $M_e > e^{tM_1}$ to exclude this trivial case. The moment problem that we consider is thus
\begin{equation} \label{eq:P-1et}
\max_{F \in \F_{1e}} \,\, \Ex_{F}[(X - q)_+],
\tag{\sf MP$_{1e}$}
\end{equation}
where $q$ is the order quantity in the newsvendor model. The dual of \eqref{eq:P-1et} is
\begin{equation} \label{eq:D-1et}
\tag{\sf DMP$_{1e}$}
\begin{aligned}
\min_z &&   z_0 +z_1M_1+z_e M_e \hspace{2cm} & \\
\text{subject to} && z_0+z_1 x+z_e e^{tx} \geq (x - q)_+,& \,\,\,\text{for all }  x \in \mathbb{R}_+.		
\end{aligned}
\end{equation}
Without loss of generality, in the rest of this section we assume that $t=1$. This is because when $t\neq 1$, we can construct $\hat X=tX$ such that $$\hat M_1=\Ex [\hat X]=t\Ex[X]=t M_1,\quad \hat M_e=\Ex [e^{\hat X}]=\Ex[e^{tX}]=M_e, \; \mbox{and}\quad  \hat q=tq,$$ and work on \eqref{eq:P-1et} with $t=1$ by 
replacing $X$ with $\hat X$.
Our unified framework allows us to arrive at the following main theorem of this section, which provides a characterization of the optimal value for problem \eqref{eq:P-1et}.
\begin{theorem} \label{tm:1et}
Suppose that $M_e>e^{M_1}$, and define 
\begin{equation}\label{eqn-v1-1et}
  v_1=-W_{-1}\left(\frac{-M_1}{M_e-1}e^{-\frac{M_1}{M_e-1}}\right)-\frac{M_1}{M_e-1}
\end{equation}
with the Lambert W function $W_{-1}(\cdot)$.  The optimal value
for problem \eqref{eq:P-1et}
is given by
\begin{equation}\label{ambiguity-1-exp}
\max_{F \in \F_{1e}} \Ex_{F}\left[\left(X - q\right)_+\right]= \begin{cases}
      M_1\left(1-\frac{q}{v_1}\right), & \text{if } \;q \leq v_1+\frac{M_1}{M_e-1}-1\\
     \frac{(v_2-q)(M_1-u)}{v_2-u}, & \text{if } \; q > v_1+\frac{M_1}{M_e-1}-1\\
    \end{cases},       
\end{equation}  
where $v_2=q+1-e^{u}\frac{M_1-u}{M_e-e^{u}}$ and $u \in \left(0,\min\left\{M_1,q\right\}\right)$ is any root of the equation
\begin{equation}\label{eqn-v-1et}
\Phi(y):=\frac{M_e-e^y}{M_1-y}(q+1-M_1)-e^y-{e^{q+1-e^{y}\frac{M_1-y}{M_e-e^{y}}}+M_e}=0.
\end{equation}
\end{theorem}
The above theorem states that the problem \eqref{eq:P-1et} has a closed-form optimal value when $q \leq v_1+\frac{M_1}{M_e-1}-1$, while in the opposite case where $q > v_1+\frac{M_1}{M_e-1}-1$, the following proposition states that Algorithm \ref{Alg: bm} can find a root of the function $\Phi(\cdot)$ defined in \eqref{eqn-v-1et}, and thus solve problem \eqref{eq:P-1et}  efficiently.
\begin{proposition}\label{prop:phi}
Suppose that the function $\Phi(\cdot)$ is defined as in \eqref{eqn-v-1et} and that $q > v_1+\frac{M_1}{M_e-1}-1$, with $v_1$ defined as in \eqref{eqn-v1-1et}. Then 
\begin{enumerate}[i)]
	\item $\Phi(0)<0$,
	\item $\Phi(q)>0$ when $M_1 > q$,
	\item $\lim_{y\uparrow M_1}\Phi(y)=+\infty>0$ when $M_1 \leq q$,
\end{enumerate}
and the bisection method in Algorithm \ref{Alg: bm}
can find a root $u \in \left(0, \min\{M_1,q\}\right)$ of  equation  \eqref{eqn-v-1et}  within $\log \left(\frac{\min\{q,M_1\}}{\epsilon }\right)$ iterations for a given precision $\epsilon$.
\end{proposition}
In the rest of this section, we sketch some major steps in the proof of Theorem \ref{tm:1et} to illustrate the idea of how our framework can be applied to solve problem \eqref{eq:P-1et}. Interested readers are referred to the appendix for detailed proofs of Theorem \ref{tm:1et} and Proposition \ref{prop:phi}.
\subsection{Sketch of Proof for Theorem \ref{tm:1et}}
To implement the first step of our framework, we  show that the  optimal solution of \eqref{eq:P-1t} is a two-point distribution and indicate where the two supporting points could possibly be allocated.
\begin{lemma}[The  rough structure of the optimal support]\label{lm:loc-1et} 
The optimal solution of \eqref{eq:P-1et} is a two-point distribution with support $\{u,v\}$ such that $0\leq u<q<v$.
\end{lemma}
To make the support set in the lemma above more explicit, we shall continue our discussion with two scenarios, one in which the supporting
point $u$ equals $0$ and one in which it is strictly greater than $0$.

\subsubsection{Results under the support $\{0,v\}$ with $v>q$.} We start with the  boundary
case of u = 0 and provide the analytical form of the optimal support as follows.

\begin{lemma}[Analytical form of the support $\{0,v\}$] \label{lm:exp-0v}
Suppose that $M_e>e^{M_1}$ and that the optimal solution of \eqref{eq:P-1et} is a two-point distribution with support $\{0,v_1\}$ such that $v_1>q$. Then $v_1$ must satisfy \eqref{eqn-v1-1et}, and $q \leq v_1+\frac{M_1}{M_e-1}-1$.
\end{lemma}
Given the analytical form of the optimal support, we can solve and verify the
optimal distribution in this case.

\begin{lemma}[Analytical form of the optimal distribution with support $\{0,v\}$] \label{lm:exp 0v-2}
Suppose that $M_e>e^{M_1}$ with $q \leq v_1+\frac{M_1}{M_e-1}-1$. Then the optimal solution of \eqref{eq:P-1et} can be characterized as
\begin{equation}
\label{opt-disn-1et-a}
X^*= \begin{cases}
      0, & w.p. \;1-  \frac{M_1}{v_1}\\
      v_1, & w.p. \; \frac{M_1}{v_1}\\
    \end{cases}  \quad \mbox{\textup{and}}\quad   \max_{F \in \F_{1e} } \Ex_{F}[(X - q)_+]=M_1\left(1-\frac{q}{v_1}\right),
\end{equation}
where $v_1$ is defined as in \eqref{eqn-v1-1et}.  
\end{lemma}

\subsubsection{Results under the support $\{u,v\}$ with $0<u<q<v$.}
Next, we discuss
the case where $u>0$ and provide the following semi-analytical form of the optimal support.

\begin{lemma}[Semi-analytical form of the support $\{u,v\}$] \label{lm:exp uv-1}
Suppose that $M_e>e^{M_1}$, $q>0$, and the optimal solution of \eqref{eq:P-1et} is a two-point distribution with support $\{u,v_2\}$ such that $0<u<q<v_2$. 
Then  $v_2=q+1-e^{u}\frac{M_1-u}{M_e-e^{u}}$ and $u\in \left(0,\min\left\{M_1,q\right\} \right)$ is a root of the equation defined in \eqref{eqn-v-1et},
and in this case we have
$q>v_1+\frac{M_1}{M_e-1}-1$ with $v_1$ defined as in \eqref{eqn-v1-1et}.
\end{lemma}
Finally, the semi-analytical form of the optimal support enables us to identify the optimal distribution as shown below.
\begin{lemma}[Semi-analytical form of the optimal distribution with support $\{u,v\}$] \label{lm:exp uv-2}
Suppose that $M_e>e^{M_1}$ and $q>v_1+\frac{M_1}{M_e-1}-1$, with $v_1$ defined in \eqref{eqn-v1-1et}.
Let 
$v_2=q+1-e^{u}\frac{M_1-u}{M_e-e^{u}}$, where $u \in \left(0,\min\left\{M_1,q\right\}\right)$ is any root of equation \eqref{eqn-v-1et}. Then $0<u<q<v_2$, and the optimal distribution for problem \eqref{eq:P-1et} can be  characterized  as
\begin{equation}\label{opt-disn-1et-b}
X^*= \begin{cases}
      u & w.p. \;\frac{v_2-M_1}{v_2-u}\\
      v_2 & w.p.  \;\frac{M_1-u}{v_2-u},\\
    \end{cases}\quad \mbox{and} \quad   \max_{F \in \F_{1e}}  \Ex_{F}[(X - q)_+]=\frac{(v_2-q)(M_1-u)}{v_2-u}.  
\end{equation}
\end{lemma}
\subsubsection{The proof of Theorem \ref{tm:1et}.}
Step $2$ and Step $3$ of our framework
have been accomplished for the two scenarios, and Theorem \ref{tm:1et} readily follows from Lemmas \ref{lm:exp 0v-2} and \ref{lm:exp uv-2}.
\section{Numerical Experiments}\label{sec:numerics}
In this section, we show that the semi-analytical solutions derived in the previous sections can also be used to numerically solve moment problems and distributionally robust problems under the newsvendor model, which provides further evidence for the capability of our proposed framework.
Since our unified framework is more of theoretical interest, the two numerical examples provided in this section are only for illustrative
purposes. We leave a deeper investigation of efficient
computational approaches for general moment problems for
future work. All of the numerical experiments in this section were performed using Matlab R2017a running on a 2.9GHz i7-7820HQ PC with 16GB memory.
\subsection{The $1$st and $t$-th moment problem in the newsvendor model} 
We first consider the $1$st and $t$-th moment problem \eqref{eq:P-1t} in Section \ref{Sec:1t} for a given $q$ (the order quantity in the newsvendor model). According to Theorem \ref{tm:1t}, this problem has a closed-form solution if $q\leq \frac{t-1}{t}M_t^{\frac{1}{t-1}}$. Otherwise, the problem cannot be solved analytically. 
In this case, a semi-analytical solution is presented in Theorem \ref{tm:1t} such that the optimal value is determined by a parameter $v$ that can be any root of the equation defined in \eqref{Eqn-Theta} and can be found using the bisection method presented in Algorithm \ref{Alg: bm}. In the following, we summarize this procedure for solving problem \eqref{eq:P-1t} in Algorithm \ref{Alg: oem1}, where {\bf BM}$(\cdot)$ is the bisection method defined in Algorithm \ref{Alg: bm}.
 \begin{algorithm}[H] 
\caption{Solve the $1$st and $t$-th moment problem} \label{Alg: oem1}
\begin{algorithmic}[1]
\State {\bf input:} moment parameters $M_1$, $M_t$, and $t$, order quantity $q$, tolerance $\epsilon>0$
\State scale $\hat M_t=\frac{M_t}{M_1^t}$  and  $\hat q=\frac{q}{M_1}$ such that $\hat M_1 =1$
\State construct function $\Theta(\cdot)$ from \eqref{Eqn-Theta} for bisection search
\If{$\hat q\leq \frac{t-1}{t}{\hat M}_t^{\frac{1}{t-1}}$}  $Z^*(\hat q)=1-\hat q {\hat M}_t^{-\frac{1}{t-1}}$ 
\Else {compute $v=${\bf BM}$\left(\Theta,\max\{{\hat M}_t^{\frac{1}{t-1}},\hat q\},\frac{t}{t-1}\hat q,\epsilon\right)$ from bisection search; }
\State \; calculate $u=\frac{t \hat q}{t-1}\frac{v^{t-1}-{\hat M}_t}{v^t-{\hat M}_t}$, and $Z^*(\hat q)=\frac{(v-\hat q)(1-u)}{v-u}$\EndIf
\State {\bf output:} $M_1\cdot Z^*(\hat q)$ ({\it perform rescaling})
\end{algorithmic}\label{Alg:1t}
\end{algorithm}
There are two other approaches that can solve the dual problem \eqref{eq:D-1t} of \eqref{eq:P-1t}. The first one is the classical SDP approach introduced by \cite{bertsimas2005optimal}. However, this approach works only if $t$ is a rational number, and its dimension also depends on $t$, which could possibly lead to a high computational cost. The second approach solves the dual problem \eqref{eq:D-1t} with the relative entropy (RE) formulation due to \cite{das2021heavy}. Since both the SDP formulation and the RE formulation are convex, we can resort to an off-the-shelf convex optimization solver such as SDPT3 (version 4.0) for an exact solution. 
\begin{figure}
\centering
    \subfloat
    {
        \begin{minipage}[t]{0.5\textwidth}
            \centering          
            \includegraphics[width=0.95\textwidth]{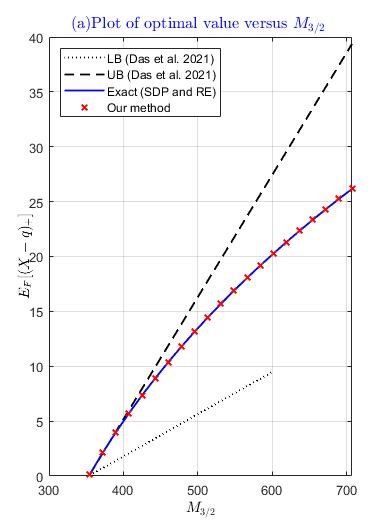}   
        \end{minipage}
    }\hspace{-10mm}
    \subfloat
    {
        \begin{minipage}[t]{0.5\textwidth}
            \centering      
            \includegraphics[width=0.95\textwidth]{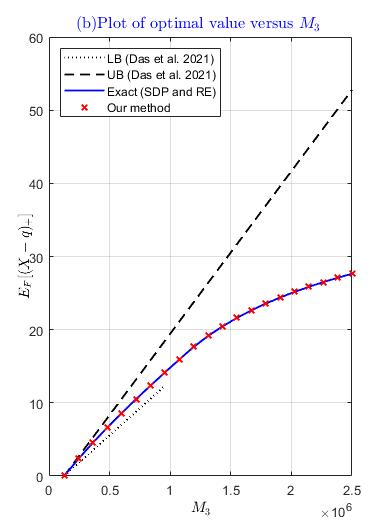}   
        \end{minipage}
    }
\caption{Given $q=100$, $M_1=50$, plot of optimal values versus $M_t$, with  $t=3/2$ for Plot(a) and $t=3$ for Plot(b). }
\label{Fg:1-b}
\end{figure}   
\begin{figure}
\centering
    \subfloat
    {
        \begin{minipage}[t]{0.5\textwidth}
            \centering          
            \includegraphics[width=0.95\textwidth]{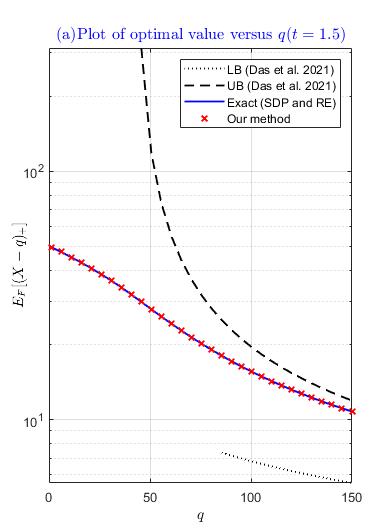}   
        \end{minipage}
    }\hspace{-10mm}
    \subfloat
    {
        \begin{minipage}[t]{0.5\textwidth}
            \centering      
            \includegraphics[width=0.95\textwidth]{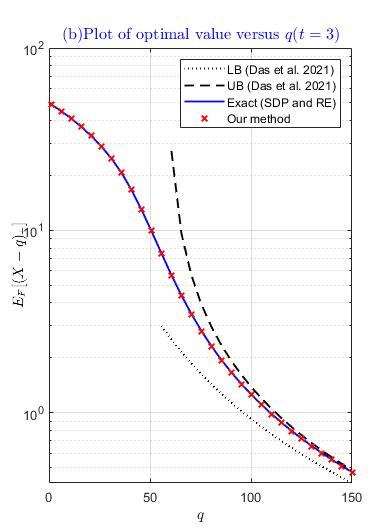}   
        \end{minipage}
    }
\caption{Plot of optimal values versus $q$ with $M_1=50$,  $M_{3/2}=1.5\times M_1^{3/2}$ for Plot(a) and $M_{3}=1.5\times M_1^3$ for Plot(b).}
\label{Fg:2}
\end{figure}  
In the numerical experiment, we apply our Algorithm \ref{Alg:1t}, the SDP approach, and the RE approach to solve the $1$st and $t$-th moment problem \eqref{eq:P-1t} with $t=3/2$ and $t=3$, and compare the optimal values computed by the three algorithms with the upper and lower bounds (abbreviated as UB and LB) from 
\cite{das2021heavy}. We set the order quantity $q=100$ and mean demand $M_1=50$, and we visualize the optimal value as a function of $M_t$ in Figure \ref{Fg:1-b}. Then we set $M_1 = 50$, $M_{\frac{3}{2}}=1.5\times M_1^{3/2}$, and $M_{3}=1.5\times M_1^3$, and we plot the optimal value as a function of $q$ in Figure \ref{Fg:2}. Both figures
indicate that the value obtained by our method is strictly larger than the lower bound and less than the upper bound by a significant margin. Note that the curves of the upper and lower bounds fail to span the entire horizontal axis, as those bounds are only valid within certain ranges of $M_t$ and $q$.
Moreover, the two figures show that the curve obtained by our approach coincides with the curves of the SDP and RE approaches. This confirms that our method can solve the moment problem \eqref{eq:P-1t} exactly. 
\begin{figure}
\centering
    \subfloat
    {
        \begin{minipage}[t]{0.5\textwidth}
            \centering          
            \includegraphics[width=0.95\textwidth]{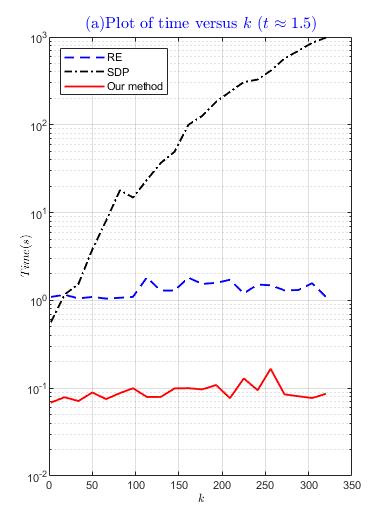}   
        \end{minipage}
    }\hspace{-10mm}
    \subfloat
    {
        \begin{minipage}[t]{0.5\textwidth}
            \centering     
            \includegraphics[width=0.95\textwidth]{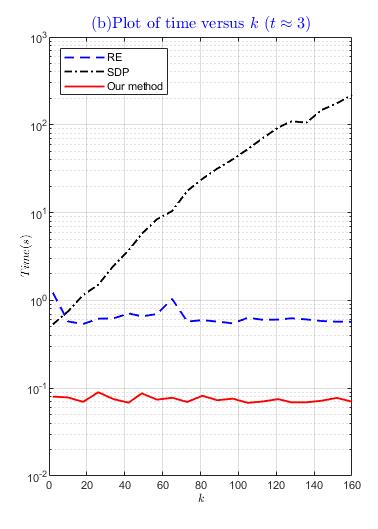}   
        \end{minipage}
    }
\caption{Given $q=100$ and $M_1=50$, plot of the running time versus $k$ with $t=\frac{1.5k+1}{k}, M_t=M_1^t+1$ for Plot(a) and $t= \frac{3k+1}{k}, M_t=M_1^t+1$ for Plot(b).}
\label{Fg:3}
\end{figure}  
Since there are three algorithms that can solve problem \eqref{eq:P-1t}, we further report their running time in Figure \ref{Fg:3} for $q=100$ and $M_1=50$. We plot the running time as a function of $k$ with $t=\frac{1.5\times k+1}{k}, M_t=M_1^t+1$ for Plot(a) and $t= \frac{3\times k+1}{k}, M_t=M_1^t+1$ for Plot(b). In this figure, it can been seen that the running time of our method and of the RE approach remain almost constant regardless of the value of $k$, while the SDP approach consumes much more time as the value of $k$ increases. This is because the dimension of the SDP is dependent on the numerator of $t$, and our choices of $t$ make such a phenomenon noticeable.   
\subsection{The distributionally robust newsvendor model with a $1$st and exponential moment ambiguity set}
In this subsection, we further demonstrate the potential of our framework by solving a 
distributionally robust newsvendor problem that cannot be tackled by either the SDP or the RE approach. 
Recall that in the newsvendor model, 
the number of
items needs to be ordered before the random demand $X$ is realized. Suppose that the unit revenue and unit purchase cost for this item are $p$ and $c$, respectively, with $p>c>0$. Given a cumulative distribution function $F$ of $X$, we want to decide on the order quantity $q$ that will maximize the total expected profit:
$$
\max_{q} \left( p\, \Ex_{F} [\min(q,X)] - c\,q  \right),
$$
where we assume that unsold units have zero salvage value. In practice, the exact demand distribution $F$ is often unknown, so here we assume that it lies in the following ambiguity set defined by the $1$st and exponential moment:
\begin{equation}\label{moment-1-e}
\F_{1e}=\left\{F \in \mathbb{M}(\R_+) : \int_{0}^\infty \mathrm{d}F(x)=1, \int_{0}^\infty x \mathrm{d}F(x)=M_1,\int_{0}^\infty e^{tx} \mathrm{d}F(x)=M_e\,\,\,\right\}.
\end{equation}
In this subsection, we consider the distributional robust newsvendor problem that maximizes the worst-case expected profit with respect to the ambiguity set $\F_{1e}$:
\begin{equation*} \label{Model: Robust}
\max_{q}\min_{F \in \F_{1e}} \left( p\, \Ex_{F} [\min(q,X)] - c\,q \right).
\end{equation*}
Using the relation $\min(q,X)=X-(X-q)_+$, this problem is equivalent to 
\begin{equation} \label{eq:P-1et-num}
\min_{q\geq 0} \max_{F \in \F_{1e}}  \left( \Ex_{F}[(X - q)_+]+(1-\eta)q \right),
\end{equation}
where $\eta=1-\frac{c}{p}$ denotes the critical ratio.
Our {approach} to problem \eqref{eq:P-1et-num} is based on the observation that the objective function $f(q):=\max_{F \in \F_{1e}}  \Ex_{F}[(X - q)_+]+(1-\eta)q$ is convex in $q$, as it is the maximum of a collection of convex functions having the form $ \Ex_{F}[(X - q)_+]$ and the convexity is preserved under the maximum operator. 
Therefore, to globally minimize $f(q)$, we can resort to the golden section search method, which requires knowledge of the function value in each step of the search procedure. Furthermore, evaluating the function value $f(q)$ for some given point $q$ is exactly the moment problem \eqref{eq:P-1et} subject to the $1$st and exponential moment constraints. According to Theorem \ref{tm:1et}, such a problem has a closed-form solution if $q\leq v_1+\frac{M_1}{M_e-1}-1$, with $v_1=-W_{-1}\left(\frac{-M_1}{M_e-1}e^{-\frac{M_1}{M_e-1}}\right)-\frac{M_1}{M_e-1}$. For other cases, a semi-closed form solution is presented in Theorem \ref{tm:1t}, where
the optimal value is determined by a parameter $u$ that can be any root of the equation defined in \eqref{eqn-v-1et}. Similar to Algorithm \ref{Alg: oem1}, we use 
the bisection method {\bf BM}$(\cdot)$   defined in Algorithm \ref{Alg: bm}
to find one root of \eqref{eqn-v-1et} and thus solve the moment problem \eqref{eq:P-1et} in Algorithm \ref{Alg: oem2}.
\begin{algorithm}[H] 
\caption{Solve the $1$st and  exponential moment problem: {\bf EM}$(M_1, M_e,t,q, \epsilon)$} \label{Alg: oem2}
\begin{algorithmic}[1]
\State {\bf input:} moment parameters $M_1$, $M_e$, and $t$, order quantity $q$, tolerance $\epsilon>0$
\State scale $\hat M_1=t{M_1}$, $\hat M_e =M_e$, and  $\hat q=t q$
\State construct function $\Phi(\cdot)$ from \eqref{eqn-v-1et} for bisection search
\State compute $v_1=-W_{-1}\left(\frac{-\hat M_1}{\hat M_e-1}e^{-\frac{\hat M_1}{\hat M_e-1}}\right)-\frac{\hat M_1}{\hat M_e-1}$
\If{$\hat q\leq v_1+\frac{\hat M_1}{\hat M_e-1}-1$}  $Z^*(\hat q)=\hat M_1(1-\frac{\hat q}{v_1})_{+}$ 
\Else {compute $u=${\bf BM}$\left(\Phi,0,\min\{\hat M_1,\hat q\},\epsilon\right)$ from bisection search; }
\State \; calculate $v_2=\hat q+1-e^u\frac{\hat M_1-u}{\hat M_e-e^u}$, and $Z^*(\hat q)=\frac{(v_2-\hat q)(\hat M_1-u)}{v_2-u}$
\EndIf
\State {\bf output:} $\frac{Z^*(\hat q)}{t}$ ({\it perform rescaling})
\end{algorithmic}
\end{algorithm}
Letting {\bf EM}$(\cdot)$ be the operator for calling Algorithm \ref{Alg: oem2}, we present the golden section search method to solve the distributionally robust newsvendor problem \eqref{eq:P-1et-num} in Algorithm \ref{Alg: oem3}. We first input an interval $[a, b]$ with $a = 0$ and $f(a) < f(b)$, where $b$ can be found by repeatedly doubling its value until the inequality holds. Since $f(\cdot)$ is convex and the optimal solution $q^* > 0$, this procedure ensures that $q^*$ is included in $[a, b]$.
Then, in each step of the while loop in Algorithm \ref{Alg: oem3}, two points $x_1$ and $x_2$ are computed based on the golden ratio. These two points divide the search interval into three subintervals. 
We evaluate the function values at $x_1$ and $x_2$ by calling the operator {\bf EM}$(\cdot)$, and
drop the rightmost or leftmost subinterval based on a comparison of their values. If the length of the remaining interval is less than $\epsilon$, the algorithm terminates and returns an $\epsilon$-optimal order quantity.

In the following experiment, we consider a classical newsvendor problem where a random demand follows a predetermined distribution that is unknown to the decision maker. We use robust models to hedge against the ambiguity in the distribution. In particular, we apply Algorithm \ref{Alg: oem3} to solve the distributionally
robust newsvendor model \eqref{eq:P-1et-num} and
compare its optimal order quantity with
two other  distributionally robust models. The ambiguity sets of these two models are defined by the first two moments (Scarf's model) and by the $1$st and $t$-th moments (the model in \citealt{das2021heavy}). 
\begin{algorithm}[H] 
\caption{Solve distributionally robust problem \eqref{eq:P-1et-num} by golden section search} \label{Alg: oem3}
\begin{algorithmic}[1]
\State {\bf input:} moment parameters $M_1$, $M_e$, and $t$, critical ratio $\eta$, interval $[a, b]$ for golden ratio
search, tolerance $\epsilon>0$
\While{true}
    \State determine two intermediate points $x_1=b-d, x_2=a+d$, with $d=\frac{\sqrt{5}-1}{2}(b-a)$ 
    \State compute the objective value at $x_1$ and $x_2$: 
     $$f^*_1=\mbox{\bf EM}(M_1, M_e,t,x_1,\epsilon)+(1-\eta)q,\; \mbox{and}\; f^*_2=\mbox{\bf EM}(M_1, M_e,t,x_2,\epsilon)+(1-\eta)q.$$
    \If{$f^*_1 \leq f^*_2$} update $a=a, b=x_2$ ({\it drop the interval $[x_2, b]$})
    \Else{update $a=x_1, b=b$} ({\it drop the interval $[ a,x_1]$})\EndIf
    \If{$\frac{b-a}{2}\leq \epsilon$}  \text{stop} and \text{\bf output:} $q^*=\frac{a+b}{2}$ \EndIf
\EndWhile
\end{algorithmic}
\end{algorithm}
We further set $t=5$ as in \cite{das2021heavy}, as the authors found that the $1$st and $5$th moments make the robust model less conservative and closer to the ground truth model than other settings in terms of the  optimal order quantity.
In addition, we provide the optimal order quantity for the ground truth model as a benchmark.
Regarding the demand distribution of the ground truth model, we consider the following three distributions, which admit different types of light tailed behavior:
\begin{itemize}
    \item Exponential random variable (density function $\phi(x)=\lambda e^{\lambda x}$ for $x>0$) with $\lambda=1/50$;
    \item Gamma random variable (density function $\phi(x)=\frac{x^k e^{-\frac{x}{\theta}}}{\Gamma(k)\theta^k}$ for $x>0$) with  $k=100$ and $\theta=0.5$;
    \item Weibull random variable (density function $\phi(x)=\frac{k}{\lambda}(\frac{x}{\lambda})^{k-1}e^{-(\frac{x}{\lambda})^{k}}$ for $x\ge 0$) with  $k = 100$ and $\lambda=50$.
\end{itemize}

The parameters for the three distributions are deliberately chosen such that they are all light tailed and all moments of finite order exist, while the exponential moment in \eqref{moment-1-e} exists only when
\begin{figure}[htbp]
\centering
    \subfloat
    {
        \begin{minipage}[t]{0.35\textwidth}
            \includegraphics[width=0.95\textwidth]{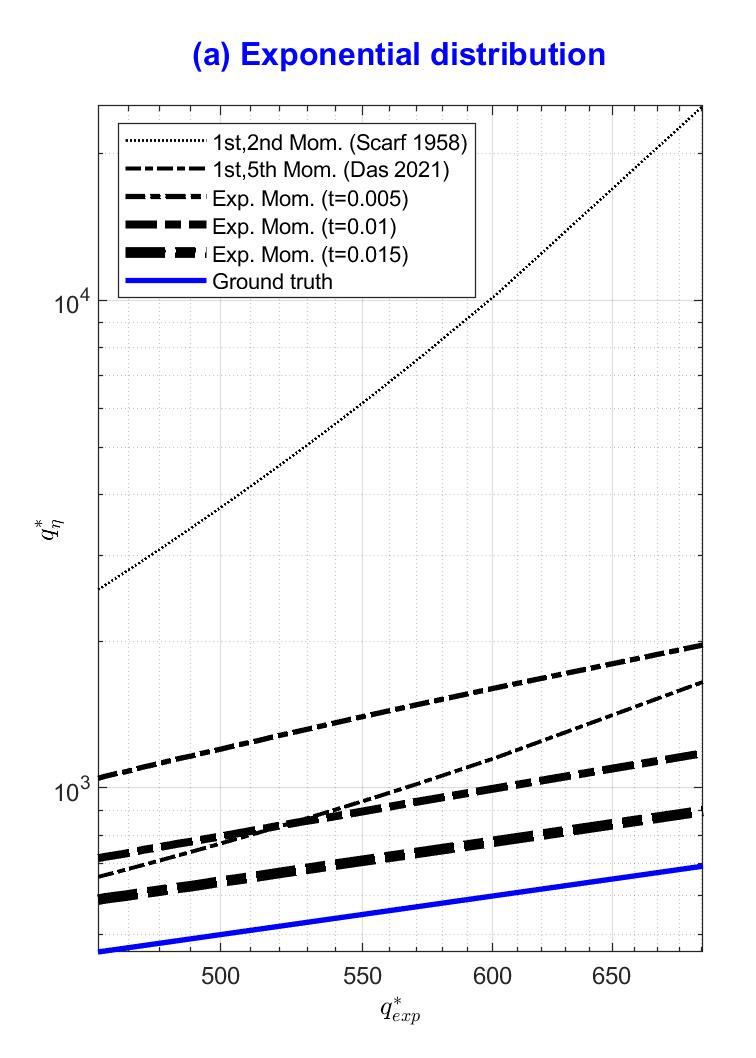} 
        \end{minipage}
    }\hspace{-10mm}
    \subfloat 
    {
        \begin{minipage}[t]{0.35\textwidth}
            \includegraphics[width=0.95\textwidth]{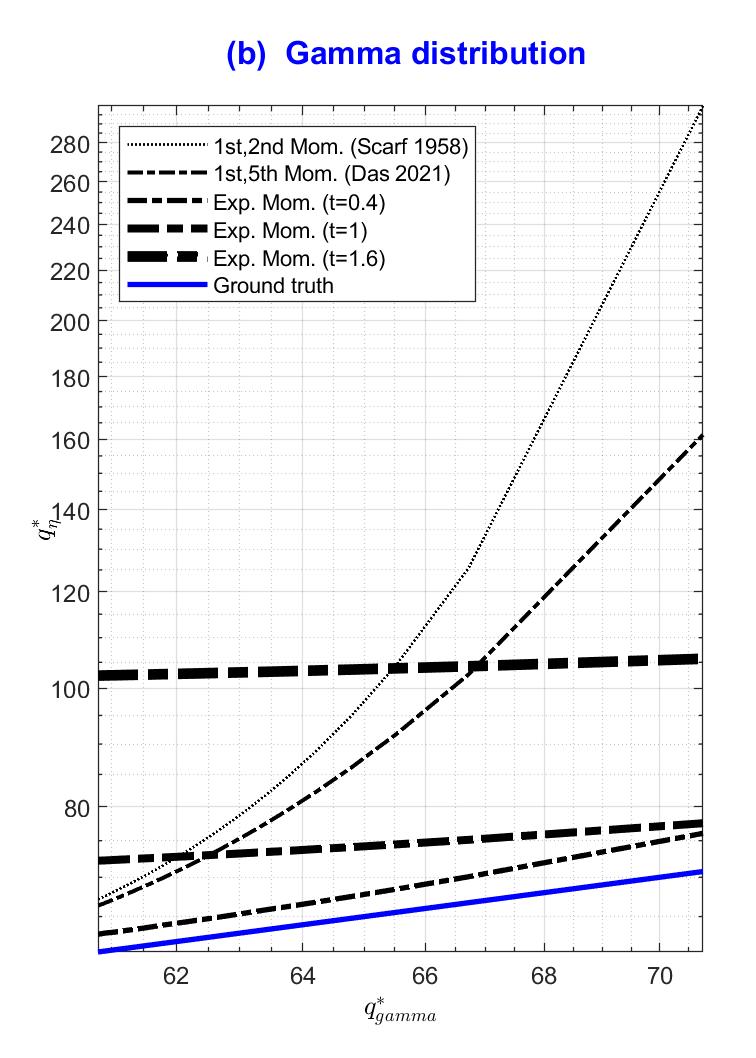}   
        \end{minipage}
    }\hspace{-10mm}
        \subfloat
    {
        \begin{minipage}[t]{0.35\textwidth}
            \includegraphics[width=0.95\textwidth]{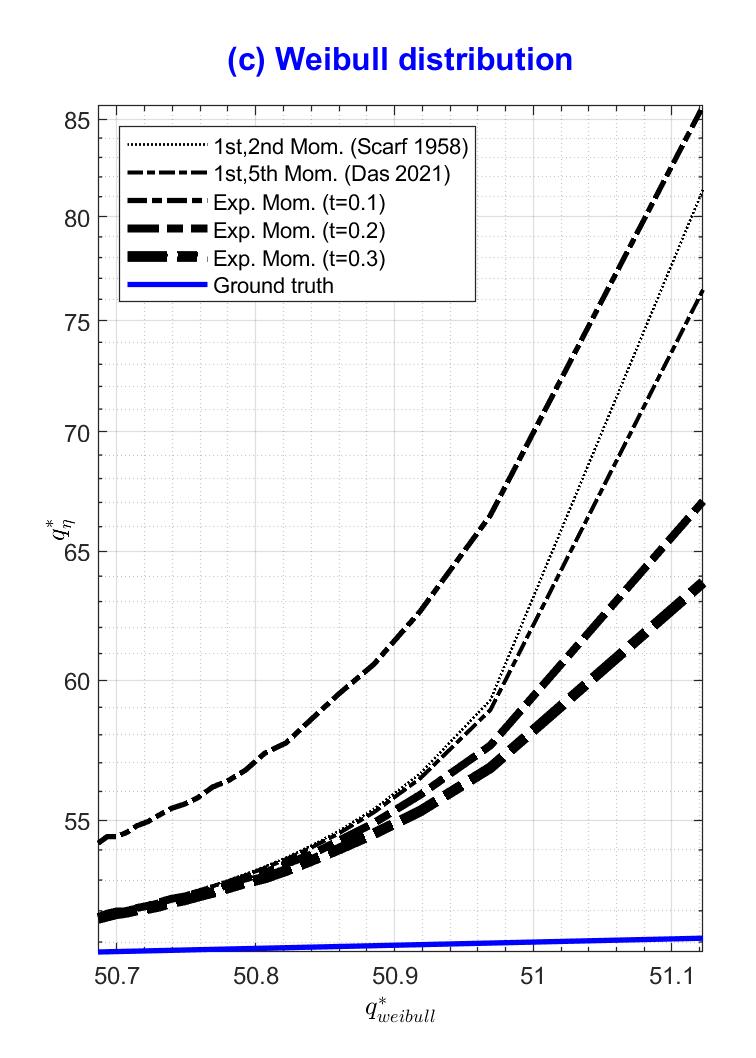}   
        \end{minipage}
    }
 \caption{Log-log plots of the optimal order quantities, where the ground truth demand follows exponential distribution in Plot(a),  gamma distribution in Plot (b), and Weibull distribution in Plot (c). }
\label{Fg:5}
\end{figure}  
$$
\left\{\begin{array}{ll} t< \lambda,& \mbox{for the exponential distribution}\\
t< \frac{1}{\theta}, & \mbox{for the gamma distribution}\\
t\in \R, & \mbox{for the Weibull distribution} \end{array} 
\right. .
$$
In Figure \ref{Fg:5}, we present the log-log plots of the optimal order quantities for the three distributionally robust newsvendor models, where the ground truth demand follows an exponential, gamma, or Weibull distribution. 
The horizontal-axis value and the vertical-axis value of every point in the subplots of Figure \ref{Fg:5}  represent, respectively, the optimal order quantities of the ground truth model and a specific distributionally robust model under the same critical ratio $\eta$.
Moreover, the critical ratio $\eta$ varies within the range $[0.9999,0.999999]$ for the exponential distribution, and within $[0.98,0.9999]$ for the gamma and Weibull distributions. In Figure \ref{Fg:5}, it can been seen that Algorithm \ref{Alg: oem3} solves the distributionally robust newsvendor problem \eqref{eq:P-1et-num} well, as it provides reasonable order quantities that are closer to the ground truth than the $1$st and $5$th moment model and Scarf's model for most instances when the critical ratio $\eta$ is large. This also indicates that incorporating exponential moment information in the model has a better chance of capturing the light tailed behavior of the underlying distribution.
\section{Conclusion}
To summarize, we propose a framework that is built on a novel primal-dual optimality condition and provides a unified treatment for generalized moment problems. Through solving three concrete moment problems, we show that this framework not only reproduces some known analytical results, but also demonstrates great potential for accommodating nonstandard moments. Even for these well-studied models, 
our framework provides new insights.
Therefore, it would be interesting to see whether there is any improvement when  our framework is applied to other existing models.
It is worth mentioning that our framework can also solve moment problems with shape constraints \citep{2005sdpconvex, perakis2008regret, van2016generalized}
that include symmetry, unimodality, convexity, and more. This is done by first using Choquet theory \citep{phelps2001lectures,2005sdpconvex} to transform these problems into problems without shape constraints and then applying the framework.
However, there are some related problems that our framework cannot handle. These problems include non-convex-shaped moment problems \citep{chen2021discrete} and discrete moment problems \citep{ninh2013discrete,prekopa2016relationship,ninh2019robust}. This is due to their inherent nonconvex natures (see, e.g., \citealt{chen2021discrete}) and the nonsmoothness of the functions in the dual constraints (see, e.g., \citealt{prekopa2016relationship}), which we leave for future work.
Another possible research direction is to extend the current framework from the univariate case to the multivariate case, as the moment problem with multidimensional random variables has more meaningful applications. One thought that immediately comes to mind is to replace the derivatives in condition \eqref{eq:Deriv} with directional derivatives
in the multivariate setting. Then, combining the optimality conditions \eqref{eq:primal}, \eqref{eq:cs}, and \eqref{eq:Deriv} will lead to
a more complicated nonlinear system, which we also leave for further study.

\section*{Acknowledgments} S. He is supported by the National Natural Science Foundation of China
[Grants 71771141,71825003 and 72192832] and the Program for Innovative Research Team of Shanghai University
of Finance and Economics. B. Jiang's research is supported by the National Natural Science Foundation of China
[Grants 72171141, 72150001 and 11831002], and
Program for Innovative Research Team of Shanghai University of Finance and Economics. Jiayi Guo is supported by the National Natural Science Foundation of China [Grant 72101139], Shanghai Sailing Program [Grant 20YF1412200], and the Fundamental Research Funds for the Central Universities.
\medskip

\bibliography{moment}

\begin{thebibliography}{}

\bibitem[Ben-Tal and Teboulle, 2007]{2010bental}
Ben-Tal, A. and Teboulle, M. (2007).
\newblock An old-new concept of convex risk measures: The optimized certainty
  equivalent.
\newblock {\em Mathematical Finance}, 17(3):449--476.

\bibitem[Bertsimas et~al., 2010]{bertsimas2010models}
Bertsimas, D., Doan, X.~V., Natarajan, K., and Teo, C.-P. (2010).
\newblock Models for minimax stochastic linear optimization problems with risk
  aversion.
\newblock {\em Mathematics of Operations Research}, 35(3):580--602.

\bibitem[Bertsimas and Popescu, 2005]{bertsimas2005optimal}
Bertsimas, D. and Popescu, I. (2005).
\newblock Optimal inequalities in probability theory: A convex optimization
  approach.
\newblock {\em SIAM Journal on Optimization}, 15(3):780--804.

\bibitem[Chebyshev, 1874]{cheb}
Chebyshev, P.~L. (1874).
\newblock Sur les valeurs limites des intégrales.
\newblock {\em Imprimerie de Gauthier-Villars}.

\bibitem[Chen et~al., 2011]{chen2011tight}
Chen, L., He, S., and Zhang, S. (2011).
\newblock Tight bounds for some risk measures, with applications to robust
  portfolio selection.
\newblock {\em Operations Research}, 59(4):847--865.

\bibitem[Chen et~al., 2021]{chen2021discrete}
Chen, X., He, S., Jiang, B., Ryan, C.~T., and Zhang, T. (2021).
\newblock The discrete moment problem with nonconvex shape constraints.
\newblock {\em Operations Research}, 69(1):279--296.

\bibitem[Chen et~al., 2019]{chen2019distributionally}
Chen, Z., Sim, M., and Xu, H. (2019).
\newblock Distributionally robust optimization with infinitely constrained
  ambiguity sets.
\newblock {\em Operations Research}, 67(5):1328--1344.

\bibitem[Das et~al., 2021]{das2021heavy}
Das, B., Dhara, A., and Natarajan, K. (2021).
\newblock On the heavy-tail behavior of the distributionally robust newsvendor.
\newblock {\em Operations Research}, 69(4):1077--1099.

\bibitem[Delage and Ye, 2010]{delage2010distributionally}
Delage, E. and Ye, Y. (2010).
\newblock Distributionally robust optimization under moment uncertainty with
  application to data-driven problems.
\newblock {\em Operations Research}, 58(3):595--612.

\bibitem[Dennis and Schnabel, 1983]{Dennis1983}
Dennis, J.~E. and Schnabel, R.~B. (1983).
\newblock {\em Numerical methods for unconstrained optimization and nonlinear
  equations}.
\newblock Prentice-Hall, Englewood Cliffs.

\bibitem[Elmachtoub et~al., 2021]{elmachtoub2021value}
Elmachtoub, A.~N., Gupta, V., and Hamilton, M.~L. (2021).
\newblock The value of personalized pricing.
\newblock {\em Management Science}, 67(10):6055--6070.

\bibitem[Fathony et~al., 2018]{fathony2018distributionally}
Fathony, R., Rezaei, A., Bashiri, M.~A., Zhang, X., and Ziebart, B.~D. (2018).
\newblock Distributionally robust graphical models.
\newblock In {\em Proceedings of the 32nd International Conference on Neural
  Information Processing Systems}, pages 8354--8365.

\bibitem[Ghaoui et~al., 2003]{ghaoui2003worst}
Ghaoui, L.~E., Oks, M., and Oustry, F. (2003).
\newblock Worst-case value-at-risk and robust portfolio optimization: A conic
  programming approach.
\newblock {\em Operations Research}, 51(4):543--556.

\bibitem[Han et~al., 2014]{han2014risk}
Han, Q., Du, D., and Zuluaga, L.~F. (2014).
\newblock A risk-and ambiguity-averse extension of the max-min newsvendor order
  formula.
\newblock {\em Operations Research}, 62(3):535--542.

\bibitem[He et~al., 2010]{he2010bounding}
He, S., Zhang, J., and Zhang, S. (2010).
\newblock Bounding probability of small deviation: A fourth moment approach.
\newblock {\em Mathematics of Operations Research}, 35(1):208--232.

\bibitem[Honorio and Jaakkola, 2014]{honorio2014tight}
Honorio, J. and Jaakkola, T. (2014).
\newblock Tight bounds for the expected risk of linear classifiers and
  {P}ac-{B}ayes finite-sample guarantees.
\newblock In {\em Artificial Intelligence and Statistics}, pages 384--392.
  PMLR.

\bibitem[Kelley, 1995]{Kelley1995}
Kelley, C.~T. (1995).
\newblock Iterative methods for linear and nonlinear equations.
\newblock In {\em Frontiers in Applied Mathematics 16}. SIAM.

\bibitem[Lanckriet et~al., 2002]{lanckriet2002robust}
Lanckriet, G.~R., Ghaoui, L.~E., Bhattacharyya, C., and Jordan, M.~I. (2002).
\newblock A robust minimax approach to classification.
\newblock {\em Journal of Machine Learning Research}, 3(Dec):555--582.

\bibitem[Li, 2018]{li2018closed}
Li, J. Y.-M. (2018).
\newblock Closed-form solutions for worst-case law invariant risk measures with
  application to robust portfolio optimization.
\newblock {\em Operations Research}, 66(6):1533--1541.

\bibitem[Markov, 1884]{markov1884certain}
Markov, A. (1884).
\newblock On certain applications of algebraic continued fractions.
\newblock {\em Unpublished Ph. D. thesis, St Petersburg}.

\bibitem[Mehrotra and Zhang, 2014]{mehrotra2014models}
Mehrotra, S. and Zhang, H. (2014).
\newblock Models and algorithms for distributionally robust least squares
  problems.
\newblock {\em Mathematical Programming}, 146(1):123--141.

\bibitem[More, 1978]{more1978}
More, J.~J. (1978).
\newblock The {L}evenberg-{M}arquardt algorithm: {I}mplementation and theory.
\newblock In {\em Lecture Notes in Mathematics 630: Numerical Analysis}, pages
  105--116. Springer-Verlag.

\bibitem[Natarajan et~al., 2010]{natarajan2010tractable}
Natarajan, K., Sim, M., and Uichanco, J. (2010).
\newblock Tractable robust expected utility and risk models for portfolio
  optimization.
\newblock {\em Mathematical Finance: An International Journal of Mathematics,
  Statistics and Financial Economics}, 20(4):695--731.

\bibitem[Natarajan et~al., 2018]{natarajan2018asymmetry}
Natarajan, K., Sim, M., and Uichanco, J. (2018).
\newblock Asymmetry and ambiguity in newsvendor models.
\newblock {\em Management Science}, 64(7):3146--3167.

\bibitem[Ninh et~al., 2019]{ninh2019robust}
Ninh, A., Hu, H., and Allen, D. (2019).
\newblock Robust newsvendor problems: Effect of discrete demands.
\newblock {\em Annals of Operations Research}, 275(2):607--621.

\bibitem[Ninh and Pr{\'e}kopa, 2013]{ninh2013discrete}
Ninh, A. and Pr{\'e}kopa, A. (2013).
\newblock The discrete moment problem with fractional moments.
\newblock {\em Operations Research Letters}, 41(6):715--718.

\bibitem[Perakis and Roels, 2008]{perakis2008regret}
Perakis, G. and Roels, G. (2008).
\newblock Regret in the newsvendor model with partial information.
\newblock {\em Operations Research}, 56(1):188--203.

\bibitem[Phelps, 2001]{phelps2001lectures}
Phelps, R.~R. (2001).
\newblock {\em Lectures on Choquet's theorem}.
\newblock Springer Science \& Business Media.

\bibitem[Popescu, 2005]{2005sdpconvex}
Popescu, I. (2005).
\newblock A semidefinite programming approach to optimal-moment bounds for
  convex classes of distributions.
\newblock {\em Mathematics of Operations Research}, 30(3):632--657.

\bibitem[Popescu, 2007]{popescu2007robust}
Popescu, I. (2007).
\newblock Robust mean-covariance solutions for stochastic optimization.
\newblock {\em Operations Research}, 55(1):98--112.

\bibitem[Pr{\'e}kopa et~al., 2016]{prekopa2016relationship}
Pr{\'e}kopa, A., Ninh, A., and Alexe, G. (2016).
\newblock On the relationship between the discrete and continuous bounding
  moment problems and their numerical solutions.
\newblock {\em Annals of Operations Research}, 238(1-2):521--575.

\bibitem[Roos et~al., 2021]{roos2021tight}
Roos, E., Brekelmans, R., van Eekelen, W., den Hertog, D., and van Leeuwaarden,
  J.~S. (2021).
\newblock Tight tail probability bounds for distribution-free decision making.
\newblock {\em European Journal of Operational Research}.

\bibitem[Rujeerapaiboon et~al., 2016]{rujeerapaiboon2016robust}
Rujeerapaiboon, N., Kuhn, D., and Wiesemann, W. (2016).
\newblock Robust growth-optimal portfolios.
\newblock {\em Management Science}, 62(7):2090--2109.

\bibitem[Scarf, 1958]{1958A}
Scarf, H. (1958).
\newblock A min-max solution of an inventory problem.
\newblock {\em Studies In The Mathematical Theory of Inventory and Production},
  (201--209).

\bibitem[Smith, 1995]{smith1995generalized}
Smith, J.~E. (1995).
\newblock Generalized {C}hebychev inequalities: {T}heory and applications in
  decision analysis.
\newblock {\em Operations Research}, 43(5):807--825.

\bibitem[Tamuz, 2013]{tamuz2013lower}
Tamuz, O. (2013).
\newblock A lower bound on seller revenue in single buyer monopoly auctions.
\newblock {\em Operations Research Letters}, 41(5):474--476.

\bibitem[Van~Parys et~al., 2016]{van2016generalized}
Van~Parys, B.~P., Goulart, P.~J., and Kuhn, D. (2016).
\newblock Generalized {G}auss inequalities via semidefinite programming.
\newblock {\em Mathematical Programming}, 156(1-2):271--302.

\bibitem[Yue et~al., 2006]{yue2006expected}
Yue, J., Chen, B., and Wang, M.-C. (2006).
\newblock Expected value of distribution information for the newsvendor
  problem.
\newblock {\em Operations Research}, 54(6):1128--1136.

\bibitem[Zuluaga et~al., 2009]{zuluaga2009third}
Zuluaga, L.~F., Pe{\~n}a, J., and Du, D. (2009).
\newblock Third-order extensions of {L}o’s semiparametric bound for
  {E}uropean call options.
\newblock {\em European Journal of Operational Research}, 198(2):557--570.

\bibitem[Zymler et~al., 2013]{zymler2013worst}
Zymler, S., Kuhn, D., and Rustem, B. (2013).
\newblock Worst-case value at risk of nonlinear portfolios.
\newblock {\em Management Science}, 59(1):172--188.

\end{thebibliography}

\ECSwitch
\ECHead{Proofs}
\section{Proofs for Section \ref{Sec:1t}}
\subsection{Proof of Proposition \ref{prop:theta}}
\begin{proof} Since $q>\frac{t-1}{t}M_t^{\frac{1}{t-1}}$ and $M_t > 1$, we have $\frac{t}{t-1}q > M_t >1$.
A straightforward calculation shows that
\begin{eqnarray*}\Theta\left(\frac{t}{t-1}q\right)&=&\frac{(\frac{tq}{t-1})^t-M_t}{\frac{tq}{t-1}-1}\left(1-\frac{(\frac{tq}{t-1})^{t}-\frac{tq}{t-1} M_t}{(\frac{tq}{t-1})^t-M_t}\right)+\left(\frac{(\frac{tq}{t-1})^{t}- \frac{tq}{t-1} M_t}{(\frac{tq}{t-1})^t-M_t}\right)^t-M_t\\
	&=& \frac{-M_t+\frac{tq}{t-1}M_t}{\frac{tq}{t-1}-1}-M_t+\left(\frac{(\frac{tq}{t-1})^{t}- \frac{tq}{t-1} M_t}{(\frac{tq}{t-1})^t-M_t}\right)^t\\
	&=&\left(\frac{(\frac{tq}{t-1})^{t}- \frac{tq}{t-1} M_t}{(\frac{tq}{t-1})^t-M_t}\right)^t>0,
\end{eqnarray*}	
and statement i) follows.	

When $q>M_t^\frac{1}{t-1}$, we let $y_q=\frac{q^t-qM_t}{q^t-M_t} \in (0,1)$, {which is equivalent to} $\frac{q^t-M_t}{q-1}=\frac{M_t}{1-y_q}$. We can express the quantity $\Theta(q)$ as 
\begin{align*}
\Theta(q)&=\frac{M_t}{1-y_q}\left(1-\frac{t}{t-1}y_q\right)+\left(\frac{t}{t-1}y_q\right)^t-M_t\\
&=-\frac{y_q}{1-y_q}\frac{M_t}{t-1}+\left(\frac{t}{t-1}y_q\right)^t\\
&=\left(\frac{t}{t-1}\right)^t\frac{y_q}{1-y_q}\left((1-y_q)y_q^{t-1}-\frac{M_t(t-1)^{t-1}}{t^t}\right).
\end{align*}
Next we define the function $\Theta_1(y)=(1-y)y^{t-1}-\frac{M_t(t-1)^{t-1}}{t^t}$, and we have $\Theta_1'(y)=y^{t-2}[(t-1)-ty]$. Then $\Theta_1(y)$ is increasing whenever $y\in (0,\frac{t-1}{t})$ and decreasing whenever $y \in (\frac{t-1}{t},1)$. Therefore, $\Theta_1(y) \leq \Theta_1(\frac{t-1}{t})=\frac{(t-1)^{t-1}}{t^t}(1-M_t)<0$ for all $y \in (0,1)$. In particular, for $y_q \in (0,1)$, we have $\Theta(q)=\left(\frac{t}{t-1}\right)^t\frac{y_q}{1-y_q}\Theta_1(y_q )<0$, which completes the proof of statement ii).

When $q \le M_t^\frac{1}{t-1}$, it is easy to verify that $\Theta\left(M_t^{\frac{1}{t-1}}\right)=0$.
{Then it remains to} calculate $\Theta'\left(M_t^\frac{1}{t-1}\right)$. To this end, we introduce two functions $\Theta_2(y)=\frac{y^t-M_t}{y-1}$
and $\Theta_3(y)=\frac{tq}{t-1}\frac{y^{t-1}-M_t}{y^t-M_t}$ such that $\Theta(y)=\Theta_2(y)\left(1 - \Theta_3(y)\right)+(\Theta_3(y))^{t}-M_t$. Since $\Theta_2\left(M_t^\frac{1}{t-1}\right)=M_t,\; \Theta_2'\left(M_t^\frac{1}{t-1}\right)=\frac{(t-1)M_t}{M_t^\frac{1}{t-1}-1}, \; \Theta_3\left(M_t^\frac{1}{t-1}\right)=0$, 
and $\Theta_3'\left(M_t^\frac{1}{t-1}\right)=\frac{tq}{M_t^\frac{1}{t-1}\left(M_t^\frac{1}{t-1}-1\right)}$, we obtain
\begin{eqnarray*}
&&\Theta'\left(M_t^\frac{1}{t-1}\right)\\
&=&\Theta_2'\left(M_t^\frac{1}{t-1}\right)\left(1-\Theta_3\left(M_t^\frac{1}{t-1}\right)\right)-\Theta_2(M_t^\frac{1}{t-1})\Theta_3'\left(M_t^\frac{1}{t-1}\right)+t\Theta_3\left(M_t^\frac{1}{t-1}\right)^{t-1}\Theta_3'\left(M_t^\frac{1}{t-1}\right) \\
&=&\frac{M_t}{M_t^\frac{1}{t-1}-1}\left(t-1-\frac{tq}{M_t^\frac{1}{t-1}}\right)<0,
\end{eqnarray*}
with the last inequality due to our assumption that $q>\frac{t-1}{t}M_t^{\frac{1}{t-1}}$. Therefore, statement iii) is also valid.

Based on the above results, we are able to apply the bisection method in Algorithm \ref{Alg: bm} to the function $\Theta(y)$
in the interval $\left( \max\{M_t^{\frac{1}{t-1}},q\},\frac{t}{t-1}q  \right)$ with a given precision $\epsilon$. In particular, when $q>M_t^\frac{1}{t-1}$, the left end point is $q$, and we have 
$\Theta\left(\frac{t}{t-1}q\right)>0$ and $\Theta\left(q\right)<0$. Then, similar to the standard bisection method, Algorithm \ref{Alg: bm}
can find a root $v$  of equation \eqref{Eqn-Theta}
within $\log \left(\frac{q}{\epsilon (t-1)}\right)$ iterations. On the other hand, when $q \le M_t^{\frac{1}{t-1}}$, the left end point is $M_t^{\frac{1}{t-1}}$, and we have 
$\Theta\left(M^{\frac{1}{t-1}}\right)=0$ with $\Theta'\left(M^{\frac{1}{t-1}}\right)<0$ and $\Theta\left(\frac{t}{t-1}q\right)>0$. Therefore, there exists $\delta > M^{\frac{1}{t-1}}$ such that $\Theta(y)<0$ for any $y \in \left(M^{\frac{1}{t-1}}, \delta \right)$. 
Hence, the updated interval after each iteration can still include at least one root of $\Theta(y)$, and 
Algorithm \ref{Alg: bm}  can still terminate
 within $\log \left(\frac{q}{\epsilon (t-1)}\right)$ iterations and return a root $v$  of equation \eqref{Eqn-Theta}.
\end{proof}
\subsection{Proof of Lemma \ref{lemma:two-root}}
\begin{proof}
Suppose that there are two points $u_1, u_2 \in [0, q)$ such that $H(u_1)=H(u_2) =0$. Since $H(x)=h_1(x)-a_1\,x -b_1$ is convex in $[0, q)$, we have, for any $x(\alpha) = \alpha u_1 + (1 - \alpha)u_2$ with $0 \le \alpha \le 1$, 
$0\le H\left(x(\alpha)\right) \le \alpha H(u_1) + (1-\alpha) H(u_2)=0$.
Thus, {we have $h_1(x) = a_1 \,x + b_1$ when $x \in [u_1,u_2]$}, which has a constant derivative $a_1$ {and} contradicts the assumption of a strictly increasing derivative of $h_1(x)$. Hence, there must be at most one root of $H(x)$ in $[0,q)$. 
Finally, by a similar argument, 
the other case (i.e., $h_2(x)$ is continuous at $q$) of conclusion (i)
as well as conclusion (ii)  follow.
\end{proof}
\section{Proofs for Section \ref{sec:P-121}}
\subsection{Proof of Lemma \ref{lm: loc-121}}
\begin{proof}
Since $M_2=\gamma M_1^2$ with $\gamma>1$, the optimal support  $S$ cannot be a singleton. Now consider the following dual constraint:
	$$
	H(x;\bz^*):=z_0^*+z_1^* x+z_2^* x^2+z_3^*(x-1)_+ - (x - 1)^2_+ \le 0,\; \forall \; x \in \mathbb{R}_+
	$$ 
	associated with the dual optimal solution  $\bz^*=(z^*_0,z^*_1,z^*_2,z^*_3)^T$. It is obvious that $z_2^* \le 1$, for otherwise the coefficient of the quadratic term is $z_2^* - 1>0$ for $x \ge 1$ and thus $ H(x;\bz^*) > 0 $ for sufficiently large $x$, which violates the dual constraint.

	According to the complementary slackness condition, the dual constraint must be tight at the {supporting} points of the primal optimal distribution. Therefore, in the following we shall identify the optimal support by searching areas where the dual constraint could be tight.
 
If $z^*_2<0$, $H(x;\bz^*)$ is concave quadratic whenever $x \in [0,1]$ or $x \in [1, +\infty)$, so $H(x;\bz^*)$ has at most one root in $[0,1]$, for otherwise, there would exist two roots $y_1$ and $y_2$ with $0\le y_1 < y_2 \le 1$ such that $H(x;\bz^*) = z^*_2 (x - y_1)(x-y_2) $ for $x \in [0,1]$. Then we would have $H(x;\bz^*)>0$ when $0\le y_1< x < y_2 \le 1$, which violates the dual constraint $H(x;\bz^*) \le 0 $ for $x \in [0,1]$. Since $H(x;\bz^*)$ is also concave quadratic whenever $x \in [1, +\infty)$, a similar argument shows that $H(x;\bz^*)$ has at most one root in $[1, +\infty)$. Suppose that $u$ and $v$ are the  
two roots of $H(x;\bz^*)$ in $[0,1]$ and $[1, +\infty)$, respectively. Then we must have 
$u \neq 1 \neq v$, for otherwise $u$ and $v$ are both included in either $[0,1]$ or $[1, +\infty)$ or the optimal distribution is single-point supported, which is a contradiction.
Therefore,  we have support $S=\{u,v\}$ with $0 \le u < 1 < v$, and prove statement i).

If $0 \leq z^*_2<1$, $H(x;\bz^*)$ is convex quadratic whenever $x \in [0,1]$ and concave quadratic whenever $x \in [1, +\infty)$. {Then, by a similar argument for that i)}, $H(x;\bz^*)$ has at most one root in $[1,+\infty)$. Since $H(x;\bz^*)$ is convex quadratic in $[0,1]$, there exist at most two roots $y_3$ and $y_4$, with $y_3 < y_4$, such that $H(x;\bz^*) = z^*_2 (x - y_3)(x-y_4) $. Then we must have $y_3, y_4 \not\in (0,1)$, for otherwise $H(x;\bz^*) > 0$ in $[0,y_3)$ or in $(y_4, 1]$, which violates the dual constraint $H(x;\bz^*) \le 0$ for $x \in [0,1]$. Therefore, $0,1$, and $v$ with $v>1$ are the only possible roots of $H(x;\bz^*)$, and $1, v$ cannot both be roots, as $H(x;\bz^*)$ has at most one root in $[1,+\infty)$. In addition, $0,1$ cannot constitute all of the roots, for otherwise the support $S = \{0, 1 \}$, which contradicts our assumption that $M_+>0$.
Therefore, $0, v$ with $v>1$ are the only roots of $H(x;\bz^*)$. That is, $S = \{0,v\}$, and statement ii) follows.

If $z^*_2=1$, then $H(x;\bz^*)=(z^*_1+z_3^*+2)x+z_0^*-z^*_3-1 \leq 0$ for all $x \geq 1$. In this case we must have $z^*_1+z_3^*+2 \le 0$, and hence $H(x;\bz^*)$ is decreasing for $x \ge 1$. Consequently, we have $H(x;\bz^*) \leq 0$ for $x \in [1, +\infty)$ as long as $H(x;\bz^*)$ is nonpositive at the left end point, i.e., $H(1;z^*)=z_0^* + z^*_1 +1 \le 0$.
In the other case where $0\le x \le 1$, we have $H(x;\bz^*)=z_0^*+z_1^* x+ x^2$. Thus, $H(x;\bz^*) \le 0$ for $x \in [0,1]$ is ensured by $H(1;\bz^*)\le 0$ and $H(0;\bz^*)=z^*_0 \le 0$. 
Since $\bz^*$ maximizes $z_0 +z_1M_1+z_3M_+$ with $M_+, M_1 >0$ in the objective function, the  {aforementioned} three inequalities on $z^*$ should be tight, i.e.,
\begin{equation} \label{eq-12-dual-opt}
    z^*_1+z_3^*+2 = 0, \quad z_0^* + z^*_1 +1 = 0,\quad z^*_0=0.
\end{equation}
That is, $H(0;\bz^*)=0$ and $H(v;\bz^*)=0$ for any $v\ge 1$. Therefore, the support $S \subseteq \{0\}\bigcup [1,+\infty)$, and statement iii) is proved.
\end{proof}
\subsection{Proof of Lemma \ref{lm: u-v-121}}
\begin{proof}
	Suppose that the optimal solution of (\ref{eq:P-121}) is supported by $\{u,v\}$ with probabilities $p_1^*$ and $p_2^*$, respectively. Then we must have
	\begin{equation}\label{eqn:uvp-121} 
	\left[
	\begin{array}{cc}
	1 & 1  \\
	u & v  \\
	u^2 & v^2 \\
	0 & v-1 \\
	\end{array}
	\right] \left[  
	\begin{array}{c}
	p_1^*\\
	p_2^* \\
	\end{array}
	\right]= \left[  
	\begin{array}{c}
	1\\
	M_1 \\
	\gamma M_1^2 \\
	M_+ \\
	\end{array}
	\right]
	\,\,\,\text{and thus} \,\,\,
	\left[  
	\begin{array}{c}
	u\\
	v\\
	p_1^*\\
	p_2^*\\
	\end{array}
	\right] =  \left[  
	\begin{array}{c}
	M_1\left(1-\frac{(\gamma-1)M_1+\kappa_{M_1,\gamma,M_+}}{2(1-M_1+M_+)}\right) \\
	M_1\left(1+\frac{(\gamma-1)M_1-\kappa_{M_1,\gamma,M_+}}{2M_+}\right) \\
	1-\frac{M_1\left(2M_+(\gamma-1)M_1+\kappa_{M_1,\gamma,M_+}\right)-2M_+}{2(1-2M_1+\gamma M_1^2)}\\
	\frac{M_1\left(2M_+(\gamma-1)M_1+\kappa_{M_1,\gamma,M_+}\right)-2M_+}{2(1-2M_1+\gamma M_1^2)}\\
	\end{array}
	\right],
	\end{equation}
	where $\kappa_{M_1,\gamma,M_+}$ {in \eqref{eqn:kappa-121}} is well defined, because
\begin{align*}
	(\gamma-1)M_1^2+4M_+(M_1-1)-4M_+^2 &=\gamma M_1^2-M_1^2+4M_+(M_1-1-M_+)\\
	& =\mathbb{V}ar(X)+4p_2^*(v-1)\left(p_1^*u+p_2^*v-1-p_2^*(v-1)\right) \\
	& =p_1^*\left(u-(p_1^*u+p_2^*v)\right)^2+p_2^*\left(v-(p_1^*u+p_2^*v)\right)^2+4p_1^*p_2^*(v-1)(u-1)\\
	& =2p_1^*p_2^*(u-v)^2+4p_1^*p_2^*(v-1)(u-1)\\
	&= 2p_1^*p_2^*\left((u-1)^2+(v-1)^2\right) >0.
	\end{align*}
Moreover, $u\ge 0$ implies that
	\begin{eqnarray}\label{u-noneg2-121}
	&&\left(2(1-M_1+M_+)-(\gamma-1)M_1\right)^2-\kappa^2(M_1,\gamma,M_+)\nonumber\\
	&=&\left(2(1-M_1+M_+)-(\gamma-1)M_1\right)^2-(\gamma-1)\left((\gamma-1)M_1^2+4M_+(M_1-1)-4M_+^2\right) \nonumber\\
	&=& (M_+-M_1+1)\left(\gamma (M_+-M_1)+1\right) \ge 0.
	\end{eqnarray}
	Note that $M_+=\mathop{\mathbb{E}}[(X-1)_+]>\mathop{\mathbb{E}}[X-1]=\mathop{\mathbb{E}}[X]-1=M_1-1$, so
	\eqref{u-noneg2-121} further implies that $M_+ \ge M_1-\frac{1}{\gamma}$, and the proof is complete. 
\end{proof}
\subsection{Proof of Lemma \ref{lm: p-z-121}}
\begin{proof} We first construct $p^*=(p_1^*, p_2^*)^T$ in accordance with \eqref{eqn:uvp-121} such that $p_1^*$ and $p_2^*$ are respectively the probabilities of $u = M_1\left(1-\frac{(\gamma-1)M_1+\kappa_{M_1,\gamma,M_+}}{2(1-M_1+M_+)}\right)$ and $v=M_1\left(1+\frac{(\gamma-1)M_1-\kappa_{M_1,\gamma,M_+}}{2M_+}\right)$ in the support of the distribution defined in
\eqref{OptDis:uv-121}. Therefore, $p_1^*, p_2^*, u, v$ satisfy $p_1^*+p_2^*=1$ and $u\, p_1^*+v\, p_2^*=M_1$ from \eqref{eqn:uvp-121}, which gives us $p_2^*=\frac{M_1-u}{v-u}$. Turning to the ranges of $u$ and $v$, 
we note from  Lemma EC.1 in \cite{han2014risk} that the feasibility of problem \eqref{eq:P-121} {implies} that $ M_1 \leq \frac{2}{\gamma}$, and thus 
\begin{equation}\label{u-noneg3-121}
2(1-M_1+M_+)-(\gamma-1)M_1 \geq 2\left(1-\frac{1}{\gamma}\right)-(\gamma-1)M_1 =(\gamma-1)\left(\frac{2}{\gamma}-M_1\right)\ge 0,
\end{equation}
where the first inequality is due to the assumption that $M_1 \le \frac{1}{\gamma} + M_+$.
Moreover, this assumption together with the argument at the end of Lemma \ref{lm: u-v-121} guarantees that \eqref{u-noneg2-121} holds, which combined with \eqref{u-noneg3-121} {yields} that
$0 \le u < M_1$.
In addition, we also have $v > M_1$, as 
\begin{equation}\label{bound:kappa}
(\gamma-1)M_1>\sqrt{(\gamma-1)\left((\gamma-1)M_1^2+4M_+(M_1-1-M_+)\right)}=\kappa_{M_1,\gamma,M_+}.
\end{equation}
Therefore,  $p_2^*=\frac{M_1-u}{v-u} \in (0,1)$, and $p_1^* \in (0,1) $ as well. That is, $\bp^*$ is a primal feasible solution of \eqref{eq:P-121}.
Next we construct the dual solution $\bz^*=(z^*_0,z^*_1,z^*_2,z_3^*)^T$ satisfying
\begin{equation}\label{eqn:z-uv-121} 
\left[  
\begin{array}{c}
z^*_0\\
z^*_1 \\
z^*_2\\
z^*_3\\
\end{array}
\right] =  \left[  
\begin{array}{c}
-\frac{1}{2}\left( \left(2M_+-M_1\right)M_1+\frac{M_1\left(2(\gamma-2)M_+^2+(\gamma-1)M_1^2-2M_+(1+(\gamma-2)M_1)\right)}{\kappa_{M_1,\gamma,M_+}}\right)\\
M_+-M_1-\frac{2M_+^2-(\gamma-1)M_1^2+M_+\left(2+(\gamma-3)M_1\right) }{\kappa_{M_1,\gamma,M_+}}\\
-\frac{1}{2}\left(\frac{(\gamma-1)M_1^2+2M_+(M_1-1)-2M_+^2}{M_1\, \kappa_{M_1,\gamma,M_+}}-1\right) \\
(M_1-1)-\frac{(\gamma-1)M_1(M_1-1-2M_+)}{\kappa_{M_1,\gamma,M_+}}\\
\end{array}
\right],
\end{equation}
which is the solution to the linear system
\begin{equation}\label{eqn:uvz-121} 
\left[
\begin{array}{cccc}
1 & u & u^2 & 0  \\
1 & v & v^2 & v-1 \\
0 & 1 & 2u  & 0 \\
0 & 1 & 2v  & 1 \\
\end{array}
\right] \left[  
\begin{array}{c}
z^*_0\\
z^*_1 \\
z^*_2\\
z^*_3\\
\end{array}
\right] = \left[  
\begin{array}{c}
0\\
(v-1)^2 \\
0 \\
2(v-1) \\
\end{array}
\right]
\end{equation}
from \eqref{eq:cs} and \eqref{eq:Deriv}.
To verify the dual feasibility of $\bz^*$,
recalling that we have proved $0\le u < M_1 < v$, we must have $u < 1$, for otherwise $M_+ =\mathop{\mathbb{E}}[(X-1)_+]=\mathop{\mathbb{E}}[X-1]=M_1-1 $ holds, which contradicts the assumption that $M_1 \le \frac{1}{\gamma} + M_+ < 1 + M_+$. In addition, the assumption $M_+>0$ implies that $v>1$. Those facts together with \eqref{eqn:uvz-121} indicate the following dual constraint function:
$$	H(x;\bz^*)=\begin{cases} z_0^*+z_1^*x+z_2^* x^2=z_2^*(x-u)^2, & \mbox{if}\;\; x \in [0,1) \\
z_0^*+z_1^*x+z_2^* x^2+z_3^*(x-1)-(x-1)^2=(z^*_2-1)(x-v)^2, & \mbox{if}\;\; x \in [1,+\infty) \end{cases}.$$
Moreover, the expression of $z_2^*$ in \eqref{eqn:z-uv-121} gives us 
\begin{eqnarray*}
z_2^*&=&-\frac{1}{2}\left(\frac{(\gamma-1)M_1^2+2M_+(M_1-1)-2M_+^2}{M_1\, \kappa_{M_1,\gamma,M_+}}-1\right)\\
&=&-\frac{1}{4} \frac{(\gamma-1)M_1^2+4M_+(M_1-1)-4M_+^2 + (\gamma-1)M_1^2 - 2M_1\, \kappa_{M_1,\gamma,M_+ }}{M_1\, \kappa_{M_1,\gamma,M_+}}\\
&=&-\frac{1}{4}\frac{\frac{\kappa_{M_1,\gamma,M_+}^2}{\gamma-1}+ (\gamma-1)M_1^2 - 2M_1\, \kappa_{M_1,\gamma,M_+ }}{M_1\, \kappa_{M_1,\gamma,M_+}} = -\frac{1}{4}\frac{\left(\kappa_{M_1,\gamma,M_+}- (\gamma-1)M_1\right)^2}{(\gamma-1)M_1\, \kappa_{M_1,\gamma,M_+}} <0,
\end{eqnarray*}
where the last inequality is due to \eqref{bound:kappa}. Therefore, $H(x;\bz^*)\le 0$ for all $x \in [0,+\infty)$, and $\bz^*$ is a dual feasible solution. In addition, $(\bp^*, \bz^*)$ also satisfies the complementary slackness condition
due to \eqref{eqn:uvz-121}, and thus we conclude that $(\bp^*, \bz^*)$ is an optimal primal-dual solution pair such that the distribution defined in \eqref{OptDis:uv-121} is optimal for \eqref{eq:P-121}.
\end{proof}
\subsection{Proof of Lemma \ref{lm: u-v-121-2}}
\begin{proof} For any feasible support $\{0,v_1,v_2,...,v_d\}$ of problem \eqref{eq:P-121} with $v_i \ge 1$ for all $1 \le i \le d$, let $p_0$ be the probability of $0$ and $p_i$ {be} the probability of $v_i$ in the support. Then the following hold:
\begin{equation*}
\sum_{i=0}^d p_i=1,\quad 0 \times p_0 +\sum_{i=1}^d p_i v_i=M_1,\quad
0 \times p_0 +\sum_{i=1}^d p_i v_i^2=\gamma M_1^2,\quad \sum_{i=1}^d p_i (v_i-1)=M_+.
\end{equation*}	 
Therefore, the associated objective value is
\begin{equation*}\sum_{i=1}^d p_i (v_i-1)^2 -M_+^2=\sum_{i=1}^d p_i v_i^2 - \sum_{i=1}^d p_i v_i  -\sum_{i=1}^d p_i (v_i-1) - M_+^2 =\gamma M_1^2-M_1-M_+-M_+^2,
\end{equation*}	
which is a constant. Since the optimal support $S\subseteq \{0\}\bigcup [1,+\infty)$, we conclude that any feasible support $\{0,v_1,v_2,...,v_d\}$ with $v_i \ge 1$ for all $1 \le i \le d$ has the same objective and thus is optimal.
In addition, a straightforward computation shows that
 \begin{equation*}
 \gamma (M_1-M_+) = \frac{\left(\sum_{i=1}^d p_i v_i^2\right)}{M_1^2}\left(\sum_{i=1}^d p_i v_i - \sum_{i=1}^d p_i (v_i-1)\right) =\frac{(\sum_{i=1}^d p_i v_i^2)(\sum_{i=1}^d p_i)}{(\sum_{i=1}^d p_iv_i)^2} \ge 1.
 \end{equation*}
The equality holds only when $v_1=v_2=...=v_d$, by the Cauchy-Schwarz inequality, and so the support reduces to $S=\{0,v\}$ with $v> 1$, {which is} the case discussed in Lemma \ref{lm: u-v-121}. Therefore, we consider the optimal support to include three points: $0,v_1,v_2$, with probabilities $p_0^*$, $p_1^*$, and $p_2^*$, respectively.
We solve the equation
\begin{equation} \label{vp-121}
\left[
\begin{array}{ccc}
1 & 1 & 1  \\
0 & v_1 & v_2  \\
0 & v_1^2 & v_2^2 \\
0 & v_1-1 & v_2-1 \\
\end{array}
\right] \left[  
\begin{array}{c}
p_1^*\\
p_2^* \\
p_3^*\\
\end{array}
\right]= \left[  
\begin{array}{c}
1\\
M_1 \\
\gamma M_1^2 \\
M_+ \\
\end{array}
\right],\quad\mbox{and obtain}\quad\left[  
\begin{array}{c}
v_2\\
p_0^*\\
p_1^*\\
p_2^*\\
\end{array}
\right]=
\left[  
\begin{array}{c}
\frac{M_1v_1-\gamma M_1^2}{(M_1-M_+)v_1-M_1}\\
1-M_1+M_+\\
\frac{M_1^2(\gamma M_1-\gamma M_+ -1)}{\gamma M_1^2 -2M_1 v_1+M_1v_1^2-M_+v_1^2}\\
\frac{(M_1v_1-M_+v_1-M_1)^2}{\gamma M_1^2 -2M_1 v_1+M_1v_1^2-M_+v_1^2}\\
\end{array}
\right].
\end{equation}
In this case, we must have $ \gamma (M_1-M_+) > 1$ and $v_2 = \frac{M_1v_1-\gamma M_1^2}{(M_1-M_+)v_1-M_1} \ge 1$, which is equivalent to $v_1 \ge \frac{\gamma M_1^2-M_1}{M_+}$. {
 As a result, we obtain} $v_1 \ge \max\left\{1, \frac{\gamma M_1^2-M_1}{M_+}\right\}$, {which completes} the proof.
\end{proof}
\subsection{Proof of Lemma \ref{lm: p-z-121-2}}
\begin{proof}	
We construct $\bp^*=(p_0^*, p_1^*, p_2^*)^T$ in accordance with \eqref{vp-121}, such that $p_0^*$, $p_1^*$, and $p_2^*$ are respectively the probabilities of $0$, $v_1$, and $v_2=\frac{M_1v_1-\gamma M_1^2}{(M_1-M_+)v_1-M_1} $ in the support of the distribution defined in
\eqref{OptDis:0v2-121}. Since we have proved that $M_1 > \frac{1}{\gamma}+M_+$ and $M_+>M_1-1$ in Lemma \ref{lm: u-v-121}, we have $0 < p_0^*<1$.
Moreover, an equivalent reformulation of $p_1^*$ shows that
$$0<p_1^* = \frac{1}{\frac{M_1-M_+}{M_1^2(\gamma M_1-\gamma M_+ -1)}\left(v_1-\frac{M_1}{M_1-M_+}\right)^2+\frac{1}{M_1-M_+}} < M_1-M_+<1,
$$
which together with $p_1^*+p_2^* = M_1-M_+$ implies that $0<p_2^*< M_1-M_+<1$ as well. Therefore, $\bp^*$ is a primal feasible solution of \eqref{eq:P-121}.
Next, we construct $\bz^*=(0,-1,1,-1)^T$ by solving \eqref{eq-12-dual-opt} and letting $z_2^*=1$. The associated dual function is
$$	H(x;\bz^*)=\begin{cases} x^2 - x \le 0, & \mbox{if}\;\; x \in [0,1) \\
0, & \mbox{if}\;\; x \in [1,+\infty) \end{cases}.$$
Consequently, $\bz^*$ is a dual feasible solution. In addition, recall that $v_1 \ge 1$, and {observe that $v_2 =\frac{M_1v_1-\gamma M_1^2}{(M_1-M_+)v_1-M_1}\ge 1$, as $v_1 \ge \frac{\gamma M_1^2-M_1}{M_+}$}. Therefore, $H(0;\bz^*)=H(v_1;\bz^*)=H(v_2;\bz^*)=0$, and the complementary slackness condition \eqref{eq:cs} holds. As a result, $(\bp^*, \bz^*)$ is an optimal primal-dual pair such that the distribution defined in \eqref{OptDis:0v2-121} is optimal for \eqref{eq:P-121}.
\end{proof}
\section{Proofs for Section \ref{sec:P-1et}}
\subsection{Proof of Proposition \ref{prop:phi}}
\begin{proof}
We start with definition \eqref{eqn-v1-1et} of $v_1$, which can be rewritten as
 $-v_1-\frac{M_1}{M_e-1}=W_{-1}(\frac{-M_1}{M_e-1}e^{-\frac{M_1}{M_e-1}})$. According to the definition of the Lambert W function, we have
$(-v_1-\frac{M_1}{M_e-1})e^{-v_1-\frac{M_1}{M_e-1}}=-\frac{M_1}{M_e-1}e^{-\frac{M_1}{M_e-1}}$, and with some simplification we obtain the equivalent formulation
\begin{equation} \label{1et:v1}
e^{v_1}=\frac{M_e-1}{M_1}v_1+1.    
\end{equation}
Noting that $M_e - M_1 -1 > e^{M_1}-M_1-1 >0$, as we assume $M_e > e^{M_1}$ and $e^{M_1}>0$, we obtain 
\begin{equation}\label{eqn:M1-Me-1}
-\frac{M_1}{M_e - 1}>-1.
\end{equation}
{Then} the strictly increasing property of $xe^{x}$ implies
$-\frac{1}{e}<\frac{-M_1}{M_e-1}e^{-\frac{M_1}{M_e-1}}<0$. Consequently, 
\begin{equation}\label{W-inverse-1}
    W_{-1}\left(\frac{-M_1}{M_e-1}e^{-\frac{M_1}{M_e-1}}\right)<W_{-1}\left(-\frac{1}{e}\right)=-1
\end{equation}
due to the Lambert W function $W_{-1}(\cdot)$ being strictly decreasing in $[-\frac{1}{e},0)$. Therefore, we have
\begin{align*}
\Phi(0)&=\frac{M_e-1}{M_1}(1+q-M_1)-1-e^{q+1-\frac{M_1}{M_e-1}}+M_e\\
&=\frac{M_e-1}{M_1}(1+q)-e^{q+1-\frac{M_1}{M_e-1}}\\
&=\frac{M_e-1}{M_1}e^{q+1}\left(\frac{1+q}{e^{1+q}} -\frac{\frac{M_1}{M_e-1}e^{v_1}}{e^{\frac{M_1}{M_e-1}+v_1}} \right)\\ &\overset{\eqref{1et:v1}}=\frac{M_e-1}{M_1}e^{q+1}\left(\frac{1+q}{e^{1+q}} -\frac{v_1+\frac{M_1}{M_e-1}}{e^{v_1+\frac{M_1}{M_e-1}}} \right)<0,
\end{align*}
where the last inequality is due to $\frac{y}{e^y}$ being a decreasing function in $(1,+\infty)$ and the assumption that $1+q>v_1+ \frac{M_1}{M_e-1}\overset{\eqref{eqn-v1-1et}}=-W_{-1}\left(\frac{-M_1}{M_e-1}e^{-\frac{M_1}{M_e-1}}\right)\overset{\eqref{W-inverse-1}}>1$. This proves statement i).\\
Now we consider the case $q< M_1$, for which we have
\begin{align*}
\Phi(q) & =\frac{M_e-e^q}{M_1-q}-e^{q+1-\frac{M_1-q}{M_e-e^q}e^q} \\
&=\frac{e(M_e-e^q)}{M_1-q}\left(e^{-1}-\frac{M_1-q}{M_e-e^q}e^{q-\frac{M_1-q}{M_e-e^q} e^q}\right)\\
&=\frac{e(M_e-e^q)}{M_1-q}\left(\left(-\frac{M_1-q}{M_e-e^q} e^q\right) e^{-\frac{M_1-q}{M_e-e^q} e^q}-(-1)e^{-1}\right).
\end{align*}
Since $e^y-y$ is a strictly increasing function for $y\geq 0$, we have $e^q-q < e^{M_1}-M_1< M_e-M_1$, which implies that $M_e-e^q> M_1-q> 0$. Therefore, $\Phi(q)$ is well defined.
{To further} prove $\Phi(q)>0$, it suffices to show that
\begin{equation}\label{x-ex>0}
\left(-\frac{M_1-q}{M_e-e^q} e^q\right) e^{-\frac{M_1-q}{M_e-e^q} e^q}-(-1)e^{-1}> 0.
\end{equation}
We first note that $e^{M_1-q} - 1> M_1-q$, as
 $e^x -1 > x$ is strictly increasing for $x>0$. 
 Multiplying both sides of this inequality by $e^q$ and recalling that $M_e> e^{M_1}$, we conclude that 
 $M_e-e^q> e^{M_1} -e^q > (M_1-q)e^q>0$, which further implies
$-\frac{M_1-q}{M_e-e^q}e^q> -1$. Therefore, \eqref{x-ex>0} follows from the observation that $x e^{x}$ is a strictly increasing function when $x\geq -1$, and hence {statement ii) is proved}.\\
Turning to the case $M_1 \leq q$, since $M_e>e^{M_1}$, we have
$$\lim_{y\uparrow M_1}\Phi(y)=\lim_{y\uparrow M_1} \left(\frac{M_e-e^y}{M_1-y}(q+1-M_1)-e^y-{e^{q+1-e^{y}\frac{M_1-y}{M_e-e^{y}}}+M_e} \right)= +\infty,$$
which completes the proof of statement iii).\\
Therefore, the function $\Phi(\cdot)$ has opposite signs at the two end points of the interval $[0, \min\{q, M_1\}]$, and the bisection method can find a root of $\Phi(\cdot)$  within $\log \left(\frac{\min\{q,M_1\}}{\epsilon }\right)$ iterations for a given precision $\epsilon$.
\end{proof}
\subsection{Proof of Lemma \ref{lm:loc-1et}}
\begin{proof} First, the optimal distribution of \eqref{eq:P-1et} cannot be supported by a single point, as we assume that $M_e > e^{M_1}$. Then the rest of the proof is similar to the proof of Lemma  \ref{lm:loc-1t}, except that the dual constraint is changed to
\begin{equation}\label{Func:H-1et-u-v}
H_{1e}(x;\bz^*):=z^*_0+z^*_1 x+z^*_e e^{x} - (x - q)_+ \ge 0,\; \forall \; x \in \mathbb{R}_+,
\end{equation}
{where} $\bz^*=(z^*_0,z^*_1,z^*_e)^{T}$ is the optimal solution of the dual problem \eqref{eq:D-1et}.
We still have $z^*_e \ge 0$, for otherwise $H_{1e}(x;\bz^*)<0$ for sufficiently large $x$,  contradicting the feasibility of $\bz^*$. When $z^*_e=0$, the dual constraint is identical to that of \eqref{eq:D-1t} with $z^*_t=0$, and hence the proof in Lemma \ref{lm:loc-1t} can be applied. Moreover, when $z^*_e>0$, the function $z^*_0+z^*_1 x+z^*_e e^{x}$ is still convex in $x$ and has a strictly positive derivative. Therefore, the conclusion follows by an argument similar to Lemma \ref{lm:loc-1t} for the case $z^*_t>0$.
\end{proof}
\subsection{Proof of Lemma \ref{lm:exp-0v}}
\begin{proof} Suppose that the optimal distribution of (\ref{eq:P-1et}) is supported by $\{0,v_1\}$ with probabilities $p_1^*$ and $p_2^*$, respectively, and that $\bz^*=(z^*_0,z^*_1,z^*_e)^T$ is an optimal solution of the dual problem (\ref{eq:D-1et}) with the constraint function $H_{1e}(\cdot)$ defined in \eqref{Func:H-1et-u-v}.
In this case, the condition \eqref{eq:Deriv} gives us ${H}_{1e}'(v_1;\bz^*)= z_1^* +  z_e^*e^{v_1}  -1 =0 $. According to Theorem \ref{tm: gpm},
$\bz^*$ and $\bp^* = (p_1^*, p_2^*)^T$ satisfy conditions \eqref{eq:primal}, \eqref{eq:cs}, and \eqref{eq:Deriv}, which together are equivalent to
\begin{equation} \label{Cond:newsvendor-1et}
\left[
\begin{array}{cc}
1 & 1   \\
0 & v_1  \\
1 & e^{v_1}\\ 
\end{array}
\right] \left[  
\begin{array}{c}
p_1^*\\
p_2^* \\
\end{array}
\right] =  \left[  
\begin{array}{c}
1\\
M_1 \\
M_e \\
\end{array}
\right]
\,\,\, \text{and  } \,\,\,
\left[
\begin{array}{ccc}
1 & 0 & 1  \\
1 & v_1 & e^{v_1} \\
0 & 1 & e^{v_1} 
\end{array}
\right] \left[  
\begin{array}{c}
z^*_0\\
z^*_1 \\
z^*_e\\
\end{array}
\right] =  \left[  
\begin{array}{c}
0\\
v_1-q \\
1 \\
\end{array}
\right].
\end{equation}
Solving the above equations gives us
\begin{equation}\label{Sol:1et}   
\frac{e^{v_1}-1}{v_1}=\frac{M_e-1}{M_1}, \quad\;
\left[  
\begin{array}{c}
p_1^* \\
p_2^* \\
\end{array}
\right]= \left[  
\begin{array}{c}
1-\frac{M_1}{v_1}\\
\frac{M_1}{v_1}\\
\end{array}
\right],
\,\,\, \text{and } \,\,\, 
\left[  
\begin{array}{c}
z^*_0\\
z^*_1 \\
z^*_e\\
\end{array}
\right]= \left[  
\begin{array}{c}
-\frac{q}{e^{v_1}(v_1-1)+1}\\
\frac{e^{v_1}(v_1-1-q)+1}{e^{v_1}(v_1-1)+1}\\
\frac{q}{e^{v_1}(v_1-1)+1}
\end{array}
\right]. 
\end{equation}
Note that the first equality above is equivalent to $e^{v_1} = \frac{M_e -1}{M_1} v_1 + 1$ and that
$$\left(-v_1-\frac{M_1}{M_e-1}\right)e^{-v_1-\frac{M_1}{M_e-1}}=-\frac{v_1+\frac{M_1}{M_e-1}}{e^{v_1+\frac{M_1}{M_e-1}}}=-\frac{v_1+\frac{M_1}{M_e-1}}{\left(\frac{M_e-1}{M_1}v_1+1\right)e^{\frac{M_1}{M_e-1}}}=-\frac{M_1}{M_e-1}e^{-\frac{M_1}{M_e-1}},$$
so $v_1=-\frac{M_1}{M_e-1}-W_{0}\left(\frac{-M_1}{M_e-1}e^{-\frac{M_1}{M_e-1}}\right)=0$ or $v_1=-\frac{M_1}{M_e-1}-W_{-1}\left(\frac{-M_1}{M_e-1}e^{-\frac{M_1}{M_e-1}}\right)$, where $W_{0}$ and $W_{-1}$ are the two branches of the Lambert W function. As our assumption indicates that $v_1>q>0$, we must have $v_1=-W_{-1}\left(\frac{-M_1}{M_e-1}e^{-\frac{M_1}{M_e-1}}\right)-\frac{M_1}{M_e-1}$, i.e., \eqref{eqn-v1-1et} holds.
In addition, as the second system of equations in \eqref{Cond:newsvendor-1et} implies $H_{1e}(0;\bz^*)=z_0^*+z_e^*=0$, the dual feasibility of $\bz^*$ requires $H_{1e}'(0;\bz^*)=z_1^*+ z_e^*=\frac{e^{v_1}(v_1-1-q)+1+q}{e^{v_1}(v_1-1)+1} \geq 0$. Note that the denominator 
\begin{equation}\label{eqn1:lm:exp-0v}
e^{v_1}(v_1-1)+1=e\left(e^{v_1-1}(v_1-1)\right)+1> e\times (-\frac{1}{e})+1= 0,    
\end{equation}
as $v_1>0$ and $x e^x$ is a strictly increasing function for $x\geq -1$.
Therefore, $H_{1e}'(0;\bz^*) \geq 0$ implies $e^{v_1}(v_1-1-q)+1+q \geq 0$, or equivalently 
$$q \leq \frac{v_1e^{v_1}}{e^{v_1}-1} -1= v_1+\frac{v_1}{e^{v_1}-1}-1\overset{\eqref{1et:v1}}=v_1+\frac{M_1}{M_e-1}-1.$$
\end{proof}
\subsection{Proof of Lemma \ref{lm:exp 0v-2}}
\begin{proof}
We construct 
$\bp^*=(p^*_1, p^*_2)^T$ and $\bz^*=(z^*_0, z^*_1, z^*_e)^T$ in accordance with \eqref{Sol:1et}, where $p^*_1$ and $p^*_2$ correspond to the probabilities associated with the support points $0$ and $v_1$ of the distribution described in \eqref{opt-disn-1et-a}. We need to show that $(\bp^*,\bz^*)$ is the optimal primal-dual solution pair for problems \eqref{eq:P-1et} and \eqref{eq:D-1et}.
First, we observe that $\frac{e^{M_1}-1}{M_1}<\frac{M_e-1}{M_1}\overset{\eqref{1et:v1}}=\frac{e^{v_1}-1}{v_1}$ due to the assumption $M_e>e^{M_1}$. {Since} the function
$\frac{e^y-1}{y}$
is strictly increasing whenever $y>0$, {we conclude} that $v_1>M_1$ and $\bp^*$ is a primal feasible solution.
Next, we verify the dual feasibility of $\bz^*$. According to the expression of $z^*_e$ in \eqref{Sol:1et} and {the inequality} \eqref{eqn1:lm:exp-0v}, we have $z^*_e>0$ {and that} $H_{1e}(x;\bz^*)$ defined in \eqref{Func:H-1et-u-v} is convex in $[0,q)$ and {in} $(q,+\infty)$, as ${H}_{1e}''(x;\bz^*)=z_e^*e^{x}>0$ for any $x \ge 0$ and $x \neq q$. Moreover, we have $H_{1e}'(0;\bz^*) \ge 0$ by the assumption that $q \leq v_1+\frac{M_1}{M_e-1}-1$ and the {argument} at the end of Lemma \ref{lm:exp-0v}. We also note $H_{1e}(0;\bz^*)=H_{1e}(v_1;\bz^*)=H_{1e}'(v_1;\bz^*)=0$ by the second system of equations in \eqref{Cond:newsvendor-1et}. It follows that
$$
{H}_{1e}(x;\bz^*) \ge \left\{\begin{array}{cl}
     H_{1e}(v_1;\bz^*) + H_{1e}'(v_1;\bz^*)(x-v_1)  = 0,& \mbox{if}\;\; x \in (q,+\infty) \\
     H_{1e}(0;\bz^*)+H_{1e}'(0;\bz^*)(x-0) \geq 0, & \mbox{if}\;\; x \in [0,q)
\end{array}
\right. ,
$$
{where we use the fact $v_1 \ge q + 1 - \frac{M_1}{M_e - 1} > q$ from \eqref{eqn:M1-Me-1}. }
Furthermore, since $H_{1e}(x;\bz^*)$ is a continuous function, $H_{1e}(q;\bz^*) \ge 0$ as well, and thus $z^*$ is a feasible solution for the dual problem \eqref{eq:D-1et}. 
In addition, $(\bp^*, \bz^*)$ also satisfies the complementary slackness condition due to the second system of equations in \eqref{Cond:newsvendor-1et}, and thus we conclude that $(\bp^*, \bz^*)$ is an optimal primal-dual solution pair. 
\end{proof}
\subsection{Proof of Lemma \ref{lm:exp uv-1}}
\begin{proof}
Suppose that the optimal distribution of (\ref{eq:P-1et}) is supported by $\{u,v_2\}$ with probabilities $p_1^*$ and $p_2^*$, respectively,
and that $\bz^*=(z^*_0,z^*_1,z^*_t)^T$ is an optimal solution of the dual problem (\ref{eq:D-1et}).
In this case, the condition \eqref{eq:Deriv} gives us ${H}_{1e}'(u;\bz^*)= z_1^* +  e^{u}  z_e^* =0 $ and ${H}_{1e}'(v_2;\bz^*)= z_1^* +  e^{v_2}  z_t^*-1 =0 $. According to Theorem \ref{tm: gpm},
$\bz^*$ and $\bp^* = (p_1^*, p_2^*)^T$ satisfy the conditions \eqref{eq:primal}, \eqref{eq:cs}, and \eqref{eq:Deriv}, which together are equivalent to
\begin{equation}  \label{eq:1-et large}
\left[
\begin{array}{cc}
1 & 1  \\
u & v_2  \\
e^{u} & e^{v_2} \\
\end{array}
\right] \left[  
\begin{array}{c}
p_1^*\\
p_2^* \\
\end{array}
\right]= \left[  
\begin{array}{c}
1\\
M_1 \\
M_e \\
\end{array}
\right]
\,\,\,\text{and } \,\,\,
\left[
\begin{array}{ccc}
1 & u & e^{u}  \\
1 & v_2 & e^{v_2} \\
0 & 1 & e^{u} \\
0 & 1 & e^{v_2} \\
\end{array}
\right] \left[  
\begin{array}{c}
z^*_0\\
z^*_1 \\
z^*_e\\
\end{array}
\right] = \left[  
\begin{array}{c}
0\\
v_2-q \\
0 \\
1 \\
\end{array}
\right].
\end{equation}
The first system of equations in \eqref{eq:1-et large} has a solution if and only if
\begin{eqnarray*}
	0 = \left|
	\begin{array}{ccc}
1 & 1  & 1\\
u & v_2  & M_1\\
e^{u} & e^{v_2} & M_e\\
	\end{array}
	\right|  \overset{\rm{row2-M_1 \cdot row1,  row3- M_e \cdot row1 }}= \left|
	\begin{array}{ccc}
1 & 1  & 1\\
u-M_1 & v_2-M_1  & 0\\
e^{u}-M_e & e^{v_2}-M_e & 0\\
	\end{array}
	\right|= \left|
	\begin{array}{cc}
u-M_1 & v_2-M_1  \\
e^{u}-M_e & e^{v_2}-M_e \\
	\end{array}
	\right|. \\
\end{eqnarray*}
As a result,
\begin{align}\label{eq:u-v-1et}
\frac{M_e-e^{u}}{M_1-u}=\frac{e^{v_2}-M_e}{v_2-M_1}=\frac{e^{v_2}-e^{u}}{v_2-u}.
\end{align}
By a similar argument, the feasibility of the second system of equations in \eqref{eq:1-et large} gives us
\begin{eqnarray}
	  0&=&\left|
	\begin{array}{cccc}
		1 & u & e^{u}  & 0\\
		1 & v_2 & e^{v_2}  & v_2-q\\
		0 & 1 &  e^{u}  & 0\\
		0 & 1 &  e^{v_2} & 1\\
	\end{array}
	\right|  \overset{\rm{row2-(v_2-q)\cdot row4}}= \left|
	\begin{array}{cccc}
1 & u & e^{u}  & 0\\
1 & q & e^{v_2}-(v_2-q) e^{v_2}  & 0\\
0 & 1 & e^{u}  & 0\\
0 & 1 & e^{v_2} & 1\\
	\end{array}
	\right|
	= \left|
	\begin{array}{ccc}
1 & u & e^{u}  \\
1 & q & e^{v_2}-(v_2-q) e^{v_2}  \\
0 & 1 & e^{u}  \\
	\end{array}
	\right| \nonumber\\
	& \overset{\rm{row2- row1}}= &\left|
	\begin{array}{ccc}
1 & u & e^{u}  \\
0 & q-u & e^{v_2}-e^{u}-(v_2-q) e^{v_2}  \\
0 & 1 & e^{u}  \\
	\end{array}
	\right| 
	 =
	\left|
	\begin{array}{cc}
 q-u & e^{v_2}-e^{u}-(v_2-q) e^{v_2}  \\
 1 & e^{u}  \\
	\end{array}
	\right|, \nonumber
\end{eqnarray}
 which further implies that $(q-u) e^{u} =e^{v_2}-e^{u}-(v_2-q) e^{v_2}$, or equivalently 
 \begin{equation} \label{eq:u-v-2et}   
  \frac{v_2e^{v_2}-ue^{u}}{e^{v_2}-e^{u}}=q+1. 
 \end{equation}
Consequently, we have 
$$q+1-v_2= \frac{v_2e^{v_2}-ue^{u}}{e^{v_2}-e^{u}}-v_2=e^{u}\frac{v_2-u}{e^{v_2}-e^{u}}\overset{\eqref{eq:u-v-1et}}=e^{u}\frac{M_1-u}{M_e-e^{u}}.$$
Therefore, 
$v_2=q+1-e^{u}\frac{M_1-u}{M_e-e^{u}}$, and we can substitute this expression of $v_2$ in $ \frac{M_e-e^u}{M_1-u}(v_2-M_1)-(e^{v_2}-M_e)$=0, which is a reformulation of \eqref{eq:u-v-1et}. Consequently, $u$ is a root of $\Phi(y)=0$.
Finally, we want to show that 
\begin{equation} \label{1et:divde2}
   q \overset{\eqref{eq:u-v-2et}}=\frac{v_2e^{v_2}-ue^{u}}{e^{v_2}-e^{u}} - 1>   \frac{v_1 e^{v_1}}{e^{v_1} -1} - 1 =v_1+\frac{v_1}{e^{v_1}-1}-1\overset{\eqref{1et:v1}}=v_1+\frac{M_1}{M_e-1}-1,
\end{equation}
and the key is to prove that the middle inequality holds.
To this end, we consider the function $\Phi_1(x, y)=\frac{xe^x-y e^y}{e^x-e^y}-1$, and for $x>y$ it is easy to verify that 
$$ \frac{\partial \Phi_1(x,y)}{\partial y}=\frac{e^{2y}}{(e^x-e^y)^2}\left((x-y-1)e^{x-y}+1\right)=\frac{e^{2y}}{(e^x-e^y)^2}\left(\sum_{i=1}^\infty \left(\frac{1}{i!}-\frac{1}{(i-1)!}\right)(x-y)^i\right)>0.$$
Hence, $\Phi_1(x, y)$ is strictly increasing with respect to $y$ if $x\ge y$, and {consequently} $\Phi_1(v_2, u) > \Phi_1(v_2, 0)$.
Moreover, $\Phi_1(x,0)$ is increasing with respect to $x$ on $(0,+\infty)$, as $\frac{\partial \Phi_1(x,0)}{\partial x}=\frac{e^x(e^x-x-1)}{(e^x-1)^2}>0$ when $x>0$. Therefore, it remains to show that $v_2 > v_1$, which gives $$\frac{v_2e^{v_2}-ue^{u}}{e^{v_2}-e^{u}} - 1 = \Phi_1(v_2,u)>\Phi_1(v_2,0)>\Phi_1(v_1,0)= \frac{v_1e^{v_1}}{e^{v_1}-1} - 1,$$ 
and the middle inequality in \eqref{1et:divde2} holds.
Hence, we construct the functions
\begin{equation}\label{eqn:phi-4-5}\Phi_2(y)=\frac{M_e-e^y}{M_1-y} \quad \mbox{and}\quad \Phi_3(y)=M_e-e^y(M_1-y+1).
\end{equation}
Noting that 
$\Phi_3(y)$ is decreasing on $[0,M_1]$ since $\Phi'_3(y)=e^y(y-M_1)\le 0$,
we have
$$\Phi'_2(y)=\frac{\Phi_3(y)}{(M_1-y)^2}\ge\frac{\Phi_3(M_1)}{(M_1-y)^2}=\frac{M_e - e^{M_1}}{(M_1-y)^2} >0,\;\mbox{when}\; 0\le y < M_1.$$
 Therefore, $\Phi_2(y)$ is a strictly increasing function, and $\frac{e^{v_2}-M_e}{v_2-M_1}=\Phi_2(v_2)\overset{\eqref{eq:u-v-1et}}=\Phi_2(u)>\Phi_2(0)=\frac{M_e-1}{M_1}$. As a result, we have
$$ \frac{e^{v_2}-1}{v_2}>\frac{e^{v_2}-1-(M_e-1)}{v_2-M_1}>\frac{M_e-1}{M_1} \overset{\eqref{1et:v1}}= \frac{e^{v_1}-1}{v_1}.$$
Since $\frac{e^y-1}{y}$ is strictly increasing when $y >0$, the above equality implies that $v_2 > v_1$.
\end{proof}
\subsection{Proof of Lemma \ref{lm:exp uv-2}}
\begin{proof}
Suppose that $u \in \left(0,\min\left\{M_1,q\right\}\right)$ is any root of equation \eqref{eqn-v-1et}, and construct $v_2=q+1-e^{u}\frac{M_1-u}{M_e-e^{u}}=q + \frac{\Phi_3(u)}{M_e -e^{u}}$. 
Since $u<M_1$ and $\Phi_3(\cdot)$ defined in \eqref{eqn:phi-4-5} is strictly decreasing, we have
$M_e-e^u(M_1-u+1)=\Phi_3(u)>\Phi_3(M_1)>0$, which yields $v_2 > q$. Moreover, recalling that $u \in \left(0,\min\left\{M_1,q\right\}\right)$,  we have $u < q < v_2$ and the following construction
\begin{equation*} 
    \left[  
\begin{array}{c}
p_1^* \\
p_2^* \\
\end{array}
\right]= \left[  
\begin{array}{c}
 \frac{v_2-M_1}{v_2-u} \\
\frac{M_1-u}{v_2-u} \\
\end{array}
\right] 
\,\,\, \text{and } \,\,\,
\left[  
\begin{array}{c}
z^*_0\\
z^*_1 \\
z^*_e\\
\end{array}
\right] =  \left[  
\begin{array}{c}
\frac{(u-1)e^{u}}{e^{v_2}-e^{u}}\\
-\frac{e^{u}}{e^{v_2}-e^{u}} \\
\frac{1}{e^{v_2}-e^{u}} \\
\end{array}
\right]
\end{equation*}
is well defined, {where} $\bp^*=(p^*_1, p^*_2)^T$ and $\bz^*=(z^*_0, z^*_1, z^*_t)^T$ are the solutions of the two linear equations in \eqref{eq:1-et large}. Therefore,
to confirm the primal feasibility, all that remains is to show that $0< p^*_1,p^*_2<1$, or equivalently, $u < M_1 < v$, due to the construction of $\bp^*$. Recall that the function $\Phi_2(\cdot)$ defined in \eqref{eqn:phi-4-5} is strictly increasing on $[0,M_1]$.
Since $u<\min\{M_1,v_2\}$ and $\Phi_2(u)\overset{\eqref{eq:u-v-1et}}=\Phi_2(v_2)$, we {must} have $v>M_1$, for otherwise, $\Phi_2(\cdot)$ cannot have the same value at $u$ and $v_2$.
Next we verify the dual feasibility of $\bz^*$.
{The} second system of equations in \eqref{eq:1-et large} indicates that
\begin{equation}\label{H1e-u-v2-0}
H_{1e}'(u;\bz^*)={H}_{1e}'(v;\bz^*)=0 \;\,\mbox{and}\;\,{H}_{1e}(u;\bz^*)={H}_{1e}(v_2;\bz^*)=0.
\end{equation}
In addition, since we showed at the very beginning of the proof that $u < q < v$, it follows that $z_e^*=\frac{1}{e^v-e^u}>0$ and ${H}_{1e}''(x;\bz^*)=z_e^*e^{x}\ge 0$ when $x\in [0,q) \bigcup (q,\infty)$. Therefore, ${H}(x;\bz^*)$ is convex in $[0,q)$ and $(q,\infty)$, which combined with \eqref{H1e-u-v2-0} gives us
$$
{H}_{1e}(x;\bz^*) \ge \left\{\begin{array}{cl}
     {H}_{1e}'(v;\bz^*)(x-v)+ {H}_{1e}(v;\bz^*)=0,& \mbox{if}\;\; x \in (q,+\infty) \\
     {H}_{1e}'(u;\bz^*)(x-u)+{H}_{1e}(u;\bz^*)=0, & \mbox{if}\;\; x \in [0,q)
\end{array}
\right. .
$$
Since ${H}_{1e}(x;\bz^*)$ is continuous at $q$, we conclude that ${H}_{1e}(x;\bz^*) \ge 0$ for all $x \ge 0$, and thus $\bz^*$ is a dual feasible solution. Observing that 
the complementary slackness condition is already implied by 
${H}_{1e}(u;\bz^*)={H}_{1e}(v;\bz^*)=0$, we conclude that
$(\bp^*, \bz^*)$ is indeed an optimal primal-dual solution pair, as desired. 
\end{proof}
\end{document}